%% file: Main.tex
\documentclass[
	,12pt%
	,pagesize%
	,headings=small%
	,paper=a4%
	,parskip=false%
	,abstract=true
	,toc=bibliography
	,DIV=12%
]{scrartcl}

\newcommand{\showbrief}{}

\addtokomafont{sectioning}{\normalfont\bfseries}
\setkomafont{title}{\normalfont}
\setkomafont{subtitle}{\normalfont}
\setkomafont{author}{\normalfont}
\setkomafont{date}{\normalfont}
\addtokomafont{pageheadfoot}{\scshape\small}
\setkomafont{caption}{\footnotesize}
\setkomafont{captionlabel}{\usekomafont{caption}\bfseries}
\setcapindent{0pt}

\usepackage{xspace}
\usepackage{subcaption} 

\usepackage[backend=biber,backref, giveninits=true,isbn=false, url=false, maxbibnames=50]{biblatex}
\usepackage[utf8]{inputenc}
\usepackage[T1]{fontenc}
\usepackage{csquotes}
\usepackage[english]{babel}
\usepackage{graphicx}
\graphicspath{{Pictures/}{/}}
\usepackage[export]{adjustbox}	

\usepackage[dvipsnames]{xcolor}
\usepackage{tikz-cd}

\usepackage[draft]{fixme}
\usepackage[nointlimits]{amsmath}
\usepackage{amssymb,mathrsfs,mathtools}
\usepackage{newtxtext}
\usepackage{newtxmath}
\AtBeginDocument{%
  \mathchardef\standardeq=\mathcode`=
  \mathchardef\standardless=\mathcode`<
  \mathchardef\standardgreater=\mathcode`>
  \mathcode`="8000
  \mathcode`<"8000
  \mathcode`>"8000
}
\begingroup\lccode`~`=\lowercase{\endgroup
  \def~}{\mathrel{\mspace{0.936mu}\standardeq\mspace{0.936mu}}}
\begingroup\lccode`~`<\lowercase{\endgroup
  \def~}{\mathrel{\mspace{0.153mu}\standardless\mspace{0.153mu}}}
\begingroup\lccode`~`>\lowercase{\endgroup
  \def~}{\mathrel{\mspace{0.153mu}\standardgreater\mspace{0.153mu}}}

\usepackage{upgreek}
\usepackage{braket}
\usepackage{cancel}
\usepackage{siunitx}
\usepackage{aliascnt} 
\usepackage[amsmath,thmmarks]{ntheorem}
\usepackage[expansion=true,protrusion=true]{microtype}
\usepackage{xkeyval}
\usepackage{authblk}

\usepackage[final,%
	pdftex,%
	bookmarks,%
	bookmarksdepth=3,%
	breaklinks=true,%
	colorlinks=true,%
	urlcolor=NavyBlue,%
	linkcolor=NavyBlue,%
	citecolor=ForestGreen,%
]{hyperref}%
\usepackage[all]{hypcap}

\AtBeginEnvironment{abstract}{\footnotesize}
\makeatletter
\patchcmd{\@maketitle}{\huge}{\Large}{}{}
\makeatother

\input{Macros}

\input{Macros_Spaces}

\usepackage[capitalise,nameinlink]{cleveref}

\newcommand{\aref}[1]{\cref{#1}}

\input{NTheoremEnglish}

\title{A Fenchel Theorem for Gauss maps and uniqueness of minimizers of nonlocal curvature energies}

\author[1]{Elias Döhrer}
\author[2]{Alexander Dohmen} 

\affil[1]{Chemnitz University of Technology, Chemnitz, Germany}

\affil[2]{RWTH Aachen University}

\begin{document}

\maketitle

\input{Abstract.tex}
 \tableofcontents
\input{Introduction}
\input{Prelim}
\input{RelatedFunctionals.tex}

\input{MinTPE.tex} 
\input{GlobaleExtremavonTP.tex}
\input{LowerLimit}
\newpage

\input{appendix}

\printbibliography
\end{document}

%% file: Macros.tex
\newcommand{\Dirichlet}{\cE}
\newcommand{\Curve}{\gamma}

\DeclareDocumentCommand{\Path}{ o }{
	\IfValueTF{#1}{%
		\varGamma_{\!#1}%
	}{%
		\varGamma%
	}%
}

\DeclareDocumentCommand{\dotPath}{ o }{
	\IfValueTF{#1}{%
		\dot{\varGamma}_{\!#1}%
	}{%
		{\dot{\varGamma}}%
	}%
}

\newcommand{\ArcLength}{L}
\newcommand{\PathLength}{\mathcal{L}}

\newcommand{\Hmeasure}{\mathcal{H}}

\DeclareDocumentCommand{\dist}{ o o o }{
	\IfValueTF{#1}{
		\varrho_{#1}\IfValueTF{#2}{(#2,#3)}{}
	}{
		\varrho\IfValueTF{#2}{(#2,#3)}{}
	}
}

\DeclareDocumentCommand{\distC}{ o o }{
	\varrho_{\gamma}\IfValueTF{#1}{(#1,#2)}{}
}

\DeclareDocumentCommand{\LineEl}{ o o }{
	\IfValueTF{#1}{
	  \omega_{#1}\IfValueTF{#2}{(#2)}{}
	}{
	  \omega\IfValueTF{#2}{(#2)}{}
	}
}

\DeclareDocumentCommand{\dLineEl}{ o o }{
	\IfValueTF{#1}{
	  \dd \omega_{#1}\IfValueTF{#2}{(#2)}{}
	}{
    \dd \omega\IfValueTF{#2}{(#2)}{}
	}
}

\DeclareDocumentCommand{\LineElC}{ o }{
	\IfValueTF{#1}{
        \omega_\Curve(#1)
	}{
        \omega_\Curve
	}
}

\DeclareDocumentCommand{\dLineElC}{ o }{
	\IfValueTF{#1}{
	  \dd \omega_\Curve(#1)
	}{
        \dd \omega_\Curve
	}
}

\DeclareDocumentCommand{\LebesgueM}{ o }{
	\IfValueTF{#1}{
        \lambda(#1)
	}{
        \lambda
	}
}

\DeclareDocumentCommand{\dLebesgueM}{ o }{
	\IfValueTF{#1}{
	  \dd \lambda(#1)
	}{
        \dd \lambda
	}
}

\DeclareDocumentCommand{\Speed}{ o }{
	\IfValueTF{#1}{
	  	h_{#1}
	}{
      	h
	}
}

\DeclareDocumentCommand{\InvSpeed}{ o }{
	\IfValueTF{#1}{
	  	H_{#1}
	}{
      	H
	}
}

\newcommand{\brief}[1]{\ifthenelse{\isundefined{\showbrief}}{}{{\color{NavyBlue}{\bigskip\emph{Brief:} #1\newline}}}}%
\newcommand{\Image}{\ensuremath{\mathrm{Im}}}

\newcommand{\Circle}{{\mathbb{S}^1}}

\newcommand{\TPE}{\mathrm{TP}}
\newcommand{\TP}{\mathrm{TP}}

\newcommand{\Angle}[2]{\measuredangle(#1,#2)}

\newcommand{\AmbDim}{n} 
\newcommand{\AmbSpace}{{\R^\AmbDim}}
\newcommand{\Domain}{{\mathbb{T}}}

\DeclareDocumentCommand{\Hess}{ O{} }{\operatorname{Hess}_{#1}}

\DeclareMathOperator{\argmin}{arg\,min}

\DeclareDocumentCommand{\converges}{ o }{
	\mathbin{%
		\IfValueTF{#1}{%
			\mathrel{\vbox{\offinterlineskip\ialign{%
				\hfil##\hfil\cr
				$\scriptscriptstyle#1$\cr
				$-\!\!\!-\!\!\!\rightarrow$\cr
			}}}
		}{%
			-\!\!\!-\!\!\!\rightarrow
		}%
	}%
}

\DeclareDocumentCommand{\wconverges}{ o }{
	\mathbin{%
		\IfValueTF{#1}{%
			\mathrel{\vbox{\offinterlineskip\ialign{%
				\hfil##\hfil\cr
				$\scriptscriptstyle#1$\cr
				$-\!\!\!-\!\!\!\rightharpoonup$\cr
			}}}
		}{%
			-\!\!\!-\!\!\!\rightharpoonup
		}%
	}%
}

\newcommand{\mymathcal}{\mathcal}

\newcommand{\cE}{{\mymathcal{E}}}

\newcommand{\dd}{\mathop{}\!\mathrm{d}}

\newcommand{\ceq}{\coloneqq}

\newcommand{\R}{{\mathbb{R}}}

\newcommand{\N}{\mathbb{N}}
\newcommand{\Z}{{\mathbb{Z}}}

\DeclarePairedDelimiterXPP{\pars}[1]{\mathop{}}{\lparen}{\rparen}{}{#1}
\DeclarePairedDelimiterXPP{\abs}[1]{\mathop{}}{\lvert}{\rvert}{}{#1}
\DeclarePairedDelimiterXPP{\norm}[1]{\mathop{}}{\lVert}{\rVert}{}{#1}
\DeclarePairedDelimiterXPP{\seminorm}[1]{\mathop{}}{\lbrack}{\rbrack}{}{#1}
\DeclarePairedDelimiterXPP{\inner}[1]{\mathop{}}{\langle}{\rangle}{}{#1}
\DeclarePairedDelimiterXPP{\iinner}[1]{\mathop{}}{\langle\!\langle}{\rangle\!\rangle}{}{#1}
\DeclarePairedDelimiterXPP{\brackets}[1]{\mathop{}}{\lbrack}{\rbrack}{}{#1}
\DeclarePairedDelimiterXPP{\braces}[1]{\mathop{}}{\lbrace}{\rbrace}{}{#1}

\DeclarePairedDelimiterXPP{\floor}[1]{\mathop{}}{\lfloor}{\rfloor}{}{#1}
\DeclarePairedDelimiterXPP{\ceil}[1]{\mathop{}}{\lceil}{\rceil}{}{#1}

\DeclarePairedDelimiterXPP{\intervalcc}[1]{\mathop{}}{\lbrack}{\rbrack}{}{#1}
\DeclarePairedDelimiterXPP{\intervalco}[1]{\mathop{}}{\lbrack}{\rparen}{}{#1}
\DeclarePairedDelimiterXPP{\intervaloc}[1]{\mathop{}}{\lparen}{\rbrack}{}{#1}
\DeclarePairedDelimiterXPP{\intervaloo}[1]{\mathop{}}{\lparen}{\rparen}{}{#1}

\DeclarePairedDelimiterXPP{\myset}[2]{\mathop{}}{\lbrace}{\rbrace}{}{#1\,\delimsize\vert\,\mathopen{}#2}

\newcommand{\fdfrac}[2]{\mbox{\footnotesize$\displaystyle\frac{#1}{#2}$}}

\newcommand{\GaussMap}{\ensuremath{\varphi}}
\newcommand{\Funct}{\ensuremath{\mathcal{I}}}


%% file: Macros_Spaces.tex
\DeclareDocumentCommand{\Graph}{ O{} O{} o o}{
	\IfValueTF{#3}{
	  \IfValueTF{#4}{
	  	\mathrm{Graph}^{#1}_{#2}(#3;#4)
	  }{
		\mathrm{Graph}^{#1}_{#2}(#3)
	  }
	}{
		\mathrm{Graph}^{#1}_{#2}
	}
}

\DeclareDocumentCommand{\Emb}{ O{} O{} o o}{
	\IfValueTF{#3}{
	  \IfValueTF{#4}{
	  	\mathrm{Emb}^{#1}_{#2}(#3;#4)
	  }{
		\mathrm{Emb}^{#1}_{#2}(#3)
	  }
	}{
		\mathrm{Emb}^{#1}_{#2}
	}
}

\DeclareDocumentCommand{\Sobo}{ O{} O{} o o}{
	\IfValueTF{#3}{
	  \IfValueTF{#4}{
	  	W^{#1}_{#2}(#3;#4)
	  }{
	  	W^{#1}_{#2}(#3)
	  }
	}{
	  W^{#1}_{#2}
	}
}

\DeclareDocumentCommand{\Bessel}{ O{} O{} o o}{
	\IfValueTF{#3}{
	  \IfValueTF{#4}{
	  	H^{#1}_{#2}(#3;#4)
	  }{
	  	H^{#1}_{#2}(#3)
	  }
	}{
	  H^{#1}_{#2}
	}
}

\DeclareDocumentCommand{\Holder}{ O{} O{} o o}{
	\IfValueTF{#3}{
	  \IfValueTF{#4}{
	  	C^{#1}_{#2}(#3;#4)
	  }{
	  	C^{#1}_{#2}(#3)
	  }
	}{
	  C^{#1}_{#2}
	}
}
\DeclareDocumentCommand{\HolderC}{ O{} O{} }{\Holder[#1][#2][\Circle][\AmbSpace]}

\DeclareDocumentCommand{\Lebesgue}{ O{} O{} o o}{
	\IfValueTF{#3}{
	  \IfValueTF{#4}{
	  	L^{#1}_{#2}(#3;#4)
	  }{
	  	L^{#1}_{#2}(#3)
	  }
	}{
	  L^{#1}_{#2}
	}
}
\DeclareDocumentCommand{\LebesgueC}{ O{} O{} }{\Lebesgue[#1][#2][\Circle][\AmbSpace]}

%% file: NTheoremEnglish.tex
\newtheorem{theorem}{Theorem}[section]
\newtheorem{lemma}[theorem]{Lemma}

\theoremstyle{break}

\theoremstyle{plain}
\theorembodyfont{\normalfont}
\newtheorem{definition}[theorem]{Definition}

{%
    \theoremsymbol{\ensuremath{\Diamond}}%
    \newtheorem{remark}[theorem]{Remark}%
}

\theoremstyle{break}

\theoremheaderfont{\itshape}
\theorembodyfont{\upshape}
\theoremstyle{nonumberplain}
\theoremseparator{.}
\theoremsymbol{\ensuremath{\Box}}	
\newtheorem{proof}{Proof}

%% file: Abstract.tex
In this paper, we prove a Fenchel theorem for Gauss maps by providing sharp lower bounds for the path length of Gauss maps of an embedding.
By combining the Fenchel-type theorem with various techniques from the field of geometric analysis, we show that circles minimize most generalized tangent-point energies. 
Furthermore, we prove that disks minimize all fractional Willmore energies 
among the class of convex planar sets. 

%% file: Introduction.tex
\section{Introduction}
Nonlocal, fractional and repulsive energies have become a very active field of research in recent years.
Repulsive energies originated in the 1980s, when Fukuhara introduced a potential of a polygonal knot, motivated by a Coulomb potential~\cite{FukuharaENERGYKNOT1988}. 
O’Hara continued on this path and defined a family of repulsive potentials (see~\cite{oharaEnergyOfKnots2-94, ohara1, ohara2}). 
In~\cite{ABRAMS2003381}, the authors proved that among curves of fixed length, the circle uniquely minimizes these potentials.
\\
Additionally, self-repulsive potentials can be used to find appealing representatives of a knot class, for example minimizers with the direct method~\cite{freedmanhewang} or critical points by means of gradient flows (\cite{Blatt-GradFlowOHara,blatt2,blatt3,FrechesSchumacherSteenebruggevonderMoselPalaisSmaleConditionGeometric2025, reiterschumacher1}).
Furthermore, such energies prove to be quite useful when modeling and simulating topological effects in physical processes (\cite{Hoidn_2002,zbMATH07990199}). 
\\
In this article, we primarily investigate \text{tangent-point energies}, which first appear in a very specific form in the work of Buck and Orloff~\cite{BuckOrloff}. 
Gonzalez and Maddocks~\cite[\S 6]{GonzalezMaddocks} suggested considering a whole family of tangent-point energies.
These energies are of special interest, because they not only exhibit self-repulsion, but they also provide a notion of fractional curvature.
For its definition let $q\geq 1$ and $\gamma\in W^{1,1}(\Domain, \R^n)$ be an immersion.
Its tangent-point energy is given by 
\begin{equation}\label{eq:DefTPq}
    \TP_q(\gamma)
    \ceq
        \int_{\Domain}
        \int_{\Domain}
            \frac{1}{r_\text{TP}[\gamma](x,y)^q}
        \abs{\gamma'(x)}
        \abs{\gamma'(y)}
        \dd x
        \dd y
    ,
\end{equation}
where 
\[
    \Domain\ceq \R/\Z 
    \;\text{ and }\;
    \frac{1}{
        r_\text{TP}[\gamma](x,y)
    }
    \ceq 
    \frac
    {
        2\abs{
            \gamma(y)-\gamma(x)
            - 
            \frac{\gamma'(x)}{\abs{\gamma'(x)}}
            \inner{
                \frac{\gamma'(x)}{\abs{\gamma'(x)}}, 
                \gamma(y)-\gamma(x)
            }
        }
    }
    {\abs{\gamma(y)-\gamma(x)}^2}
\]
denotes the inverse of the radius of the smallest circle passing through $\gamma(x)$ and $\gamma(y)$ while also being tangent to $\gamma'(x)$ at $\gamma(x)$.

In~\cite{strzeleckivdm} it was shown that the self-avoidance property holds true for $q\geq 2$. 
More precisely, an admissible curve with finite length and finite energy has to be embedded; see~\cite[Theorem 1.1]{strzeleckivdm}.
In~\cite{strzeleckivdm2} these energies were also generalized to suitable $k$-dimensional subsets of $\R^n$,
still exhibiting the self-avoidance property for $q>2k$, even for the more difficult scale invariant case $q=2k$~\cite{Kfer:820789}.
The regularizing effect has been proven by Strzelecki and von der Mosel (see~\cite[Theorem 1.3]{strzeleckivdm},~\cite[Theorem 1.4]{strzeleckivdm2}).
Blatt proved in~\cite{zbMATH06214305} that the energy space of $\TP_q$ is given by the fractional Sobolev-Slobodeckĳ \;space $W^{2-k/q,q}$.
Regarding applications, these energies are excellent for modeling impermeability, see~\cite{MR4182084,MR4273107,MR3800032,10.1145/3658174,10.1145/3478513.3480521,2006.07859}.
\\
Since the energy space $W^{2-\frac{1}{q},q}$ for curves embeds into $C^1$ only if $q>2$, this excludes working in Hilbert-spaces, if one needs that embedding.
This motivated Blatt and Reiter to decouple the powers of the numerator and the denominator to introduce~\cite{blattreiter1} the \textit{generalized tangent-point energies} $\TPE^{(p,q)}$. 
Let $p,q\in \intervalco{1,\infty}$ and $\gamma\in W^{1,1}(\Domain, \R^n)$ be an immersion. 
We define 
\begin{equation}\label{eq:DefTPpq}
    \TPE^{(p,q)}(\gamma)
        \ceq
            \int_{\Domain}
            \int_{\Domain}
                \frac{\abs{P^\perp_{\gamma^\prime(u)}(\gamma(u+w)-\gamma(u))}^q}
                {\abs{\gamma(u+w)-\gamma(u)}^p}
            \abs{\gamma'(u)}
            \abs{\gamma'(u+w)}
            \dd u
            \dd w
    ,
\end{equation}
where $P^\perp_v$ denotes the projection onto the orthogonal complement of $\text{span}(v)$ for $v \in \R ^n \setminus \{0\}$.
In the case $p=2q$, the generalized tangent-point energies simplify to the classical functional from \aref{eq:DefTPq}, i.e. $\TPE^{(2q,q)}=2^{-q} \, \TP_q$, hence the case $p=2q$ is referred to as the \textit{geometric} case.
For $q>1$ and $p\in \intervalco{q+2,2q+1}$, the energy $\TPE^{(p,q)}$ exhibits self-repulsion with energy space $W^{(p-1)/q,q}(\Domain, \R^n)$~\cite[Theorem 1.1]{blattreiter1}. 
In the case $q=2$, Blatt and Reiter exploit the quadratic structure of the energy and show that critical points of $\TPE^{(p,2)}$ (constrained to curves parametrized by arc length) are smooth~\cite[Theorem 1.5]{blattreiter1}.
In cooperation with Reiter and Schumacher the first author of this paper adapted their analysis in order to define a complete Riemannian metric on the space of embedded curves~\cite{DohrerReiterSchumacherCompleteRiemannianMetric2025}.
Additionally, the first author of this paper and Freches showed that $\TPE^{(p,2)}, p\in (4,5)$ are real analytic~\cite{DoehrerFreches-ConvergenceOfGradFlowsOnKnottedCurves}.
\\
The energy is still well-defined for $q\in \intervalcc{0,1}$ or $ p\in \intervalco{0,q+2}$, but does not exhibit the self-avoidance property.
Nonetheless, for $p >q+1$ it penalizes certain types of self-intersections, hence we refer to these energies as \textit{mildly repulsive}, see \aref{section: The lower limit case p=q+1} for details.
In contrast, for $q >1, p \geq 2q+1$, the energy is infinite among embedded closed curves.\\
Additionally, we investigate minimizers of the closely related fractional Willmore energies, introduced in~\cite{blatt2025fractional}.
These energies arise from taking the $L^p$-norm of a nonlocal version of the mean curvature, defined in~\cite{CaffarelliSavinRoquejoffreNonLocalMinSurfaces}.
The first study of the subcritical case has been carried out in~\cite{blatt2025fractional}.
More recently, in~\cite{giacomin2024convex}, the authors investigated the critical, i.e.\ scale-invariant case and proved that minimizers exist in the class of convex sets in $\R^2$.
For $s\in (0,1), p\geq 1$, the fractional Willmore energy of a sufficiently nice set $E\subset \R^2$, where $\partial E$ is parametrized by $\gamma: \Domain \rightarrow \R^2$, is given by 
\begin{equation}\label{eq:DefnNonLocalWillmoreE}
    \mathcal{W}_{s,p}(E)
        \ceq
    \int_{\Domain}
        \abs*{
            c_s\int_{\Domain}
            \frac{\inner{n_\gamma(y), \gamma(x)-\gamma(y)}}{\abs{\gamma(x)-\gamma(y)}^{2+s}}
            \abs{\gamma'(y)}
            \dd y
        }^p
        \abs{\gamma'(x)}
        \dd x
    ,
\end{equation}
where $n_\gamma(y)$ denotes the inward pointing unit normal of the image of $\gamma$ at $\gamma(y)$ and $c_s$ is a constant.
\\
In the following, we denote the length of a curve $\gamma:\Domain \rightarrow \R^n$ by
\begin{align*}
	\PathLength(\gamma)\ceq \int_\Domain \abs{\gamma'}\dd t,
\end{align*}
and define the following subsets of Sobolev spaces
\begin{align*}
    W^{1,p}_{\mathrm{i,r}}(\Domain, \R^n)&\ceq \{ \gamma \in W^{1,p}(\Domain, \R^n)| \, \gamma \text{ injective and regular}\} , \quad p \in \intervalcc{1, \infty},
    \\
    W^{1,\infty}_{\mathrm{i,a}}(\Domain, \R^n)&\ceq \{ \gamma \in W^{1,\infty}_{\mathrm{i,r}}(\Domain, \R^n)| \, \abs{\gamma'}=1\}
    .
\end{align*}
Furthermore, for an arbitrary Riemannian manifold $(M,g)$, we define the geodesic distance
\[
\dist[M][x][y]\ceq
\inf \{ \int_0^1 \abs{f'(t)}\dd t \;: f\in C^1([0,1],M), f(0)=x, f(1)=y\}
.
\]
We investigate the global minimizers of the \textit{generalized tangent-point energies} on $W^{1,\infty}_{\mathrm{i,a}}$ and the \textit{fractional Willmore energies} in the class of convex sets in $\R^2$.
Due to the homogeneity of the functionals, one has to constrain the energy, for example, by fixing the length of the curve, respectively the perimeter of the set. 
While finding minimizers of knot energies in a given isotopy class proves to be a more delicate task, the global minimizer is often conjectured to be the round circle. 
A knot energy is called \textit{basic}, if it is uniquely (up to similarities) minimized by circles among closed, injective curves of fixed length. 
So far, basicness has been established for most O'Hara energies; see~\cite[Corollary 12]{ABRAMS2003381}, for some nonlocal curvature energies related to the so-called \textit{integral Menger curvature}; see~\cite[Proposition 3.1]{strzelecki2013some},~\cite[Lemma 7]{strzelecki2007rectifiable} or for an analogous result for higher dimensional sets~\cite[Theorem 1.5]{KolasinskiStrzVdM-W2pSubmanifoldsByPIntegrability}, and for the geometric tangent-point energies $\TP_q$; see~\cite[Corollary 5.12]{volkmann1}. 
The main purpose of this paper is to extend the last result to the generalized tangent-point energies for a wide range of parameters $p$ and $q$.

\subsection{Main Results}
We prove, that the generalized tangent-point energies $\TPE^{(p,q)}$ are basic, for a wide range of parameters $p$ and $q$.
The statement is split into two theorems. 
The first theorem provides a sharp lower bound for $\TPE^{(p,q)}$, while the second theorem constitutes a rigidity statement for curves attaining said lower bound.
\begin{theorem}[Sharp lower bound for $\TP^{(p,q)}$]\label{thm: main theorem TP}
    Let $q\geq 1$, $p\in \intervalco{q+1,2q+1}\cap \intervalcc{2q-2,4q-2}$ or $p=q+1$ and $L>0$.
    For every $\gamma\in W^{1,1}_{\mathrm{i,r}}(\Domain, \R^n)$ with $\PathLength(\gamma)=L$, the following holds true.
    \begin{equation}\label{eqn:SharpLowerBoundForTP}
        \TP^{(p,q)}(\gamma)
        \geq
        L^{q+2-p} \, \pi^{p-q} 
        \int_{0}^1 
            \sin (\pi w)^{2q-p}
        \dd w
        =
        L^{q+2-p} \, \pi^{p-q} \frac{\Gamma (\frac{2q-p+1}{2})}{\sqrt{\pi} \,  \Gamma (\frac{2q-p+2}{2})}
    ,
    \end{equation}
    where $\Gamma$ denotes the Gamma function.
    Additionally, for $q\geq 1, p\in \intervaloo{2q,2q+1}$, the inequality holds among all convex, Lipschitz curves.
    Equality in~\eqref{eqn:SharpLowerBoundForTP} holds if $\Image(\gamma)$ is a circle.
\end{theorem}
The set of parameters, for which the first part and the second part of \aref{thm: main theorem TP} hold true, is depicted in Figure 1 (b) as the blue and the red region, respectively.
Surprisingly, the theorem also holds true for $q\geq 1, p=q+1$.
For comparison, the yellow region in Figure 1 (a) illustrates the regime of parameters for which $\TP^{(p,q)}$ exhibits self repulsion.
The discussion of equality in~\eqref{eqn:SharpLowerBoundForTP} is carried out in the following theorem.
\begin{theorem}[Uniqueness of minimizers]\label{thm: UniquenessTheoremTP}
    Let $\gamma\in W^{1,1}_{\mathrm{i,r}}$ with $\PathLength(\gamma)=L$ be given,
    such that $\gamma$ realizes equality in~\eqref{eqn:SharpLowerBoundForTP}.
    \begin{itemize}
        \item[(i)] If $q=1,p=2$, then $\gamma$ is convex.
        \item[(ii)] If $q>1$ and $p\in \intervalco{q+1,2q+1}\cap \intervalcc{2q-2,4q-2}$ or $p=q+1$, then $\Image(\gamma)$ is a circle.
        \item[(iii)] If $q\geq 1, p\in (2q,2q+1)$ and $\gamma$ is convex, then $\Image(\gamma)$ is a circle.
    \end{itemize}
\end{theorem}
Note that we obtain uniqueness of circles as minimizers in the entire blue region and in the red region (among convex curves) in Figure 1(b), with $p=2,q=1$ being the only exception.
\newpage
\begin{figure}[!ht]
\begin{subfigure}{0.45\textwidth}
  \centering
  \includegraphics[width=\linewidth]{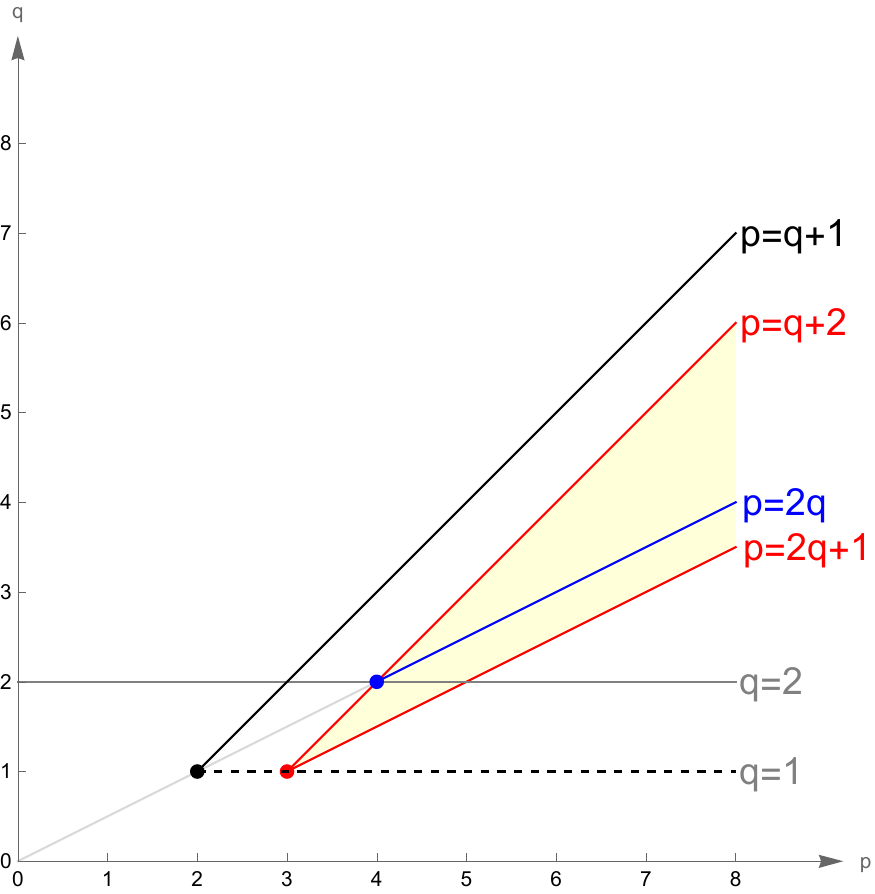}
  \caption{Self-repulsive regime of $\TP^{(p,q)}$}
\end{subfigure}
\hfill
\begin{subfigure}{0.45\textwidth}
  \centering
  \includegraphics[width=\linewidth]{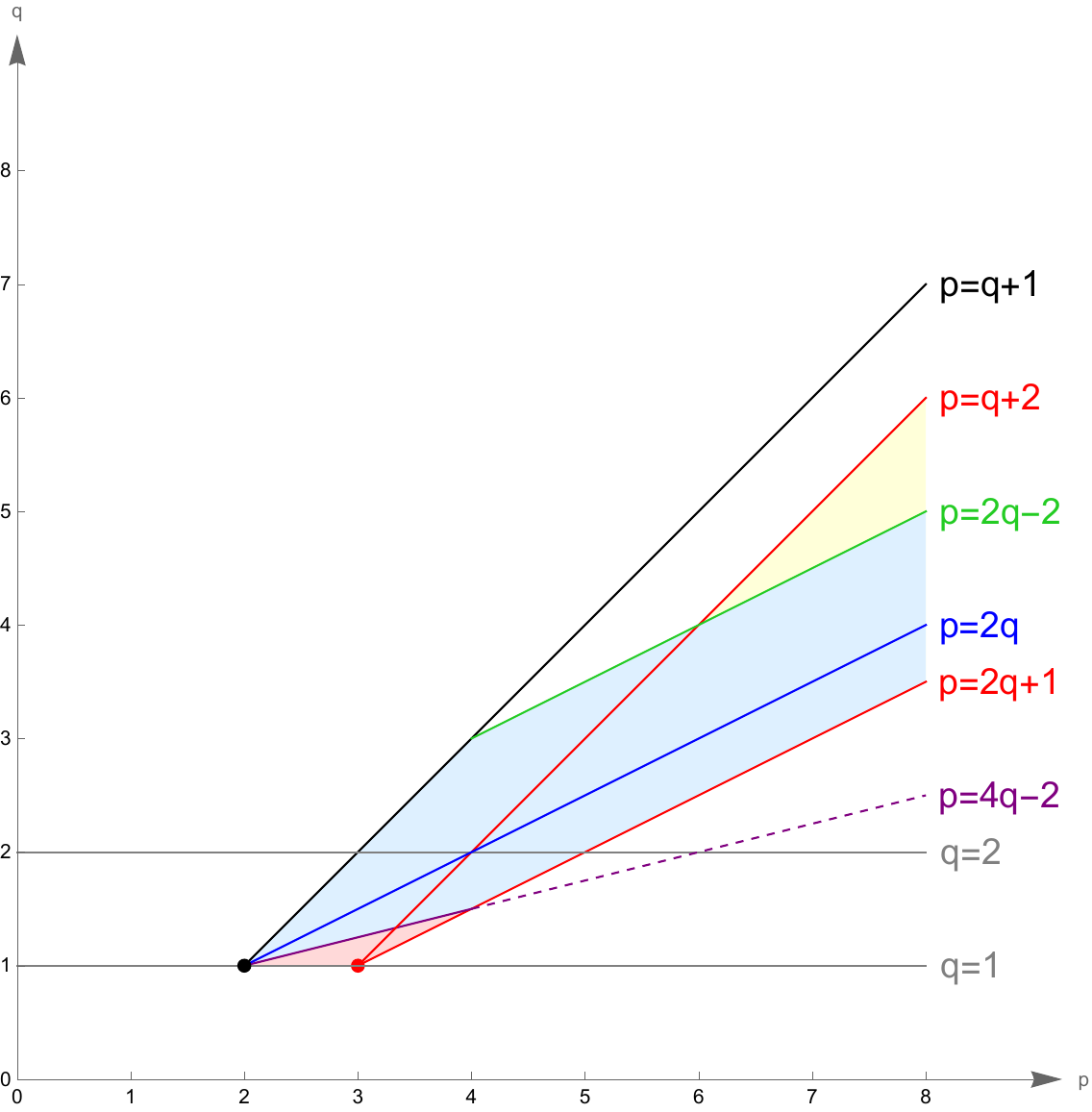}
  \caption{areas where circles minimize $\TP^{(p,q)}$}
\end{subfigure}
\caption{
Figure 1 (a) illustrates the self-repulsive regime of $\TP^{(p,q)}$.
It consists of the scale-invariant \textit{critical case} $p=q+2, q>1$ and the \textit{subcritical case} $p\in (q+2,2q+1), q>1$ (yellow region).
For $p\geq 2q+1$, the energy of any closed $C^1_{\mathrm{i,r}}$-curve is infinite.
We introduce the line $p=q+1, q\geq 1$, which will be referred to as \textit{lower limit case}.
Even though $\TP^{(p,q)}$ is not a knot energy for $p\in \intervaloo{q+1,q+2}, q\geq 1$, the energy still penalizes certain types of self-intersections.
Hence, we refer to this as \textit{mild repulsion}; see \aref{section: The lower limit case p=q+1} for more details.
\\
Figure 1 (b) indicates regions where $\TP^{(p,q)}$ is minimized by circles (blue and red region).
Note that for $p<q+1$, there is no global minimizer within the class of closed, injective, $W^{1,1}$-curves with fixed length, since the infimum $0$ is only attained for straight lines.
}
\end{figure}
\noindent
Motivated by our findings concerning $p=q+1$, we conjecture that circles globally minimize $\TP^{(p,q)}$ for all $q\geq 1, p\in \intervalco{q+1,2q+1}$. 
However, within the framework developed in this paper, one can only show the following:
there exists a $c^\ast\in \intervalcc{2, 2.5}$ such that the scope of \aref{thm: main theorem TP} and \aref{thm: UniquenessTheoremTP} extends to $p \in \intervalco{q+1,2q+1}\cap \intervalcc{2q-c^\ast,2q+c^\ast(q-1)}$; see \aref{rmk:LimitsOfOurMethods}.
\\

As a consequence of Theorems~\ref{thm: main theorem TP} and~\ref{thm: UniquenessTheoremTP}, we prove that disks uniquely minimize the fractional Willmore energies, among convex sets with fixed perimeter.
\begin{theorem}\label{thm:NonlocalWillmoreMinByCircle}
    Let $s\in \intervaloo{0,1}, p\geq 1$. 
    Then for all convex $E \subset \R^2$ with $\text{per}(E)=L$, the following holds true.
    \begin{align*}
        \mathcal{W}_{s,p}(E)
        \geq
        L^{1-ps}
        \pars*{\pi^{1+s}\int_0^1 \sin(\pi w)^{-s} \dd w}^p
    ,
    \end{align*}
    with equality iff $E$ is a disk.
\end{theorem}
Theorems~\ref{thm: main theorem TP} and~\ref{thm: UniquenessTheoremTP} rely on what we call a \textit{Fenchel theorem for Gauss maps}.
For $\gamma\in W^{1,1}_{\mathrm{i,r}}(\Domain, \R^n)$, we define the Gauss map of $\gamma$ as
\begin{equation}\label{eq:DefGaussMap}
    \GaussMap_{\gamma}(u,w)
    \ceq
    \frac{\gamma(u+w)-\gamma(u)}{\abs{\gamma(u+w)-\gamma(u)}}
    .
\end{equation}
Fixing $w=w_0>0$ and varying $u$ can be interpreted as fixing a distance in the parametrization of the curve and investigating the direction of the secant.
Since $\GaussMap_{\gamma}(u,w)$ is an approximation of the unit tangent, a lower bound on the length of $u\mapsto \GaussMap_\gamma(u,w)$
can be interpreted as an approximation of Fenchel's Theorem about the total curvature; see~\cite[\S~5.7, Theorem~3]{DoCarmo-DiffGeoCurvesAndSurfaces} for the standard Fenchel Theorem.
\begin{theorem}[Fenchel Theorem for Gauss Maps]\label{thm: Approximate Fenchel Theorem}
    Let $\gamma\in W^{1,1}_{\mathrm{i,r}}(\Domain, \R^n)$.
    Then the following holds.
    \[
    	\llap{(i) \hspace{3.5cm}}
        \int_0^1 \abs{\partial_u \GaussMap_\gamma(u,w)}\dd u \geq 2\pi 
        \;\text{ for all } w\in \intervaloo{0,1}, 
    \]
    with equality for all $w$ if and only if $\gamma$ is convex.
    \[
    \llap{(ii) \hspace{3.3cm}}
        \int_0^1 \abs{\partial_w \GaussMap_\gamma(u,w)}\dd w\geq \pi
        \;\text{ for almost all }u\in \Domain,
    \]
    with equality for almost all $u$ if and only if $\gamma$ is convex.
\end{theorem}
Because we do not require $\gamma$ to be parametrized by arc length, the quantity $\GaussMap_\gamma$ cannot be considered purely geometric.
We stick to the more involved quantity $\GaussMap$, because on the one hand, the lower bounds naturally hold true for a larger class of curves. 
On the other hand, as seen in \aref{lem: constant derivative of GaussMap}, the property $\abs{\partial_u \GaussMap_\gamma} \equiv 2\pi$ implies additional regularity.
\\
The connection between $\TP^{(p,q)}$ and $\PathLength(w\mapsto \GaussMap_\gamma(u,w))$ was observed by S. Blatt in~\cite[\S 5]{volkmann1} as
\begin{equation}\label{eq:DerivGaussMapW}
    \abs{\partial_w \GaussMap_\gamma(u,w)}
    =\frac{\abs{P^\perp_{\GaussMap_\gamma}\gamma'(u+w)}}{\abs{\gamma(u+w)-\gamma(u)}}
    =\frac{\abs{\gamma'(u+w)}}{2\mathrm{r}_{\mathrm{TP}}(\gamma)(u+w,u)}
    .
\end{equation}
A key insight for our analysis here, is to consider the derivative with respect to $u$ as well, namely
\begin{equation}\label{eq:DerivGaussMapU}
    \abs{\partial_u \GaussMap_\gamma(u,w)}=
    \frac{\abs{P^\perp_{\GaussMap_\gamma}(\gamma'(u+w)-\gamma'(u))}}{\abs{\gamma(u+w)-\gamma(u)}}
    .
\end{equation}
We proceed as follows.
The lower bounds in~\aref{thm: Approximate Fenchel Theorem} (i) and (ii) are derived in~\aref{sec:WirtingerAndLengthBounds} in Lemmas~\ref{lem:LengthOfGaussMapinW} and Lemmas~\ref{lem:LengthOfGaussMapinU}, respectively. 
In~\aref{sec:RelFunct}, we provide a collection of functionals, which, as a consequence of \aref{thm: Approximate Fenchel Theorem}, are minimized by convex curves or circles.
We dedicate \aref{sec:MinMostTPE} to finding suitable minorants for the tangent-point energies, which are special cases of the functionals treated in~\aref{sec:RelFunct}. 
Theorems~\ref{thm: main theorem TP} and~\ref{thm: UniquenessTheoremTP} then follow as a combination of \aref{thm: circle min p geq 2q} and \aref{thm: circle minimizer for p>=2q-2} by using the invariance under reparametrization of $\TP^{(p,q)}$ and  homogeneity, i.e. $\TP^{(p,q)}(\lambda \gamma)= \lambda^{q+2-p}\TP^{(p,q)}(\gamma)$.
In~\aref{sec:NonlocalWillmore}, we prove \aref{thm:NonlocalWillmoreMinByCircle} by providing suitable minorants in terms of certain tangent-point functionals.

%% file: Prelim.tex
\section{Preliminaries}
This chapter is dedicated to collecting and deriving some fundamental inequalities, which are needed in the following.
The first three inequalities have been used in~\cite{ABRAMS2003381}, where the authors minimize most O'Hara energies.
There, the authors investigate functionals of the form 
\begin{equation}
    f(c)
    =
        \iint_{\Domain\times \Domain}
            F(\abs{c(s)-c(t)},\dist[\Domain](s,t))
        \dd(s,t),
\end{equation}
where $c:\Domain \rightarrow \R^n$ is a curve, parametrized by arc length.
Their theorem, put into our notation, reads as follows.
\begin{theorem}[Abrams, Cantarella, Fu, Ghomi, Howard,~\cite{ABRAMS2003381}]
    If $F:\R^2 \rightarrow \R$ is convex and decreasing in $(x_1)^2\in (0, x_2^2)$ for $x_2\in \intervaloo{0,\frac{1}{2}}$,
    then $f$ is uniquely minimized by the round circle among all closed, unit-speed curves.
\end{theorem}
A central ingredient of their analysis is a Wirtinger-type inequality.
It states the following.
\begin{theorem}[Theorem $5$,~\cite{ABRAMS2003381}]\label{thm: Wirtinger type inequality}
   Let $c:\Circle:= \mathbb{R}/2 \pi \mathbb{Z} \rightarrow \mathbb{R}^n$ be an absolutely continuous function. If $c^\prime(t)\in \Lebesgue[2]$, then for any $s \in \mathbb{R}$
   \begin{equation}
        \int_{\Circle}
            \abs{c(t+s)-c(t)}^2 
        \dd t 
        \leq 
            (
                2\sin(\frac{s}{2})
            )^2
            \int_{\Circle}
                \abs{c^\prime (t)}^2 
            \dd t
   \end{equation}
   with equality if and only if $s$ is an integral multiple of $2 \pi$ or $\gamma$ is a circle, parametrized by unit speed.
\end{theorem}
We may apply \aref{thm: Wirtinger type inequality} to arc length parametrized curves, as done similarly in~\cite{ABRAMS2003381}, and extend the statement to lower exponents.
\begin{lemma}\label{lem: circle max of Delta gamma |^r}
    Let $w\in \intervaloo{0,1}$, $\gamma \in W^{1,\infty}_{\mathrm{i,a}}(\Domain, \R^n)$ and $g:[0,\frac{1}{2}]\rightarrow\intervalco{0,\infty}$ be concave and strictly monotonically increasing.
    Then 
    \begin{equation}
        \int_{\Domain}
            g(\abs{\gamma(u+w)-\gamma(u)}^2)
        \dd u
        \leq
        g(
            \frac{\sin(\pi w)^2}{\pi^2}
        )
    \end{equation}
    with equality iff $\Image(\gamma)$ is a circle.
\end{lemma}
\begin{proof}
    Let $\gamma\in W^{1,\infty}_{\mathrm{i,a}}$.
    We use \aref{thm: Wirtinger type inequality} and conclude that
    \[
        \int_{\Domain}
            \abs{\gamma(u+w)-\gamma(u)}^2
        \dd t
        \leq
        \frac{\sin(\pi w)^2}{\pi ^2}
    \]
    with equality iff $\Image(\gamma)$ is a circle.
    The claim now follows by applying Jensen's inequality for concave functions onto $g$.
\end{proof}
We now prove a version of the above for functions, mapping to the unit sphere.
\begin{lemma}\label{lem:WirtingerMapsOnSphere}
    Let $f\in W^{1,1}(\Domain, \mathbb{S}^{n-1})$ with $\PathLength(f)\leq 2\pi$ and $w\in \intervaloo{0,1}$.
    Then
    \[
        \int_\Domain
            \abs{f(u+w)-f(u)}
        \dd u 
        \leq
        2\sin ( \pi w)
        ,
    \]
    with equality iff $\Image(f)$ is a great circle and $\abs{f'(u)}=2\pi$ for almost every $u\in \Domain$.
\end{lemma}
\begin{proof}
	We will prove the equivalent bound
	\[
		\int_\Domain
		    \abs{f(u+w)-f(u)}
		\dd u 
		\leq
		2\sin ( \pi \abs{w}) 
        \quad 
        \text{for } 
        w \in [- \frac{1}{2}, \frac{1}{2}] \setminus \{0 \}
        .
	\]
    Let $w \in [- \frac{1}{2}, \frac{1}{2}] \setminus \{0 \} $. 
    We can bound the distance $\dist[\mathbb{S}^n]$ on $\mathbb{S}^n$ of $f(u+w)$ and $f(u)$ by
    \[
    0 
    \leq 
        \dist[{\mathbb{S}^n}](f(u+w),f(u)) 
    \leq 
        \min \{ 
            \abs{w} 
            \int_0^{1} 
                \abs{f^\prime (u+ \theta w)} 
            \dd \theta, 
            \, \pi 
        \} 
    \leq 
        \pi.
    \]
    We define
    \[
        \lambda: [0,\pi] \rightarrow \mathbb{R} \, , 
        \quad \lambda (s) := 2 \, \sin ( \frac{s}{2})
    \]
    and observe, that one can relate $\abs{f(u+w)- f(u)}$ to their distance on the sphere by
    \[
        \abs{f(u+w)-f(u)}
        =
        \dist[\AmbSpace](f(u+w),f(u))
        = 
        \lambda \big( 
            \dist[{\mathbb{S}^n}](f(u+w),f(u))
        \big)
        .
    \]
    Since $\lambda$ is monotonically increasing and concave on $[0,\pi]$, 
    we can apply Jensen's inequality in order to estimate the integral over the secant distance.
    \begin{align*}
        \int _{\Domain} 
            \abs{f(u+w)-f(u) } 
        \dd u 
    &=
        \int _{\Domain}  
            \lambda \big( 
                \dist[{\mathbb{S}^n}] (f(u+w),f(u))
            \big) 
        \dd u
    \\
    &\leq
        \lambda \left( 
            \int _{\Domain}  
                \dist[\mathbb{S}^n](f(u+w),f(u))  
            \dd u 
            \right)
    \end{align*}
    Once again, the above terms are equal, if and only if $u\mapsto \dist[\mathbb{S}^n](f(u+w),f(u))$ is constant.
    Using that $\lambda$ is monotonically increasing, we conclude that
    \begin{align*}
        \int _{\Domain} 
            \abs{f(u+w)-f(u) } 
        \dd u 
    &\leq
        \lambda \pars*{
            \int _{\Domain} 
            \min \left\{
                \abs{w}\, 
                \int_0^{1} 
                    \abs{f^\prime (u+ \theta w)} 
                \dd \theta,
                 \, \pi 
            \right\}\, du 
        }
    \\
    &\leq
        \lambda \pars*{
            \min \left\{ 
                \int _{\Domain} 
                    \abs{w} 
                    \int_0^{1} 
                        \abs{f^\prime (u+ \theta w)} 
                    \dd \theta
                \dd u , 
                \int _{\Domain} \pi \dd u \right\} 
        }
    \\
    &=
        \lambda \left( 
            \min \left\{ 
                \abs{w} 
                \int_0^{1} 
                    \int _{\Domain}
                        \abs{f^\prime (u+ \theta w)} 
                    \dd u 
                \dd \theta, 
                \pi 
            \right\} 
        \right)
    \\
    &=
        \lambda \left( 
            \min \left\{ 
                \abs{w} 
                \PathLength(f),
                \pi \right\} 
        \right)
    \\
    &\leq
        \lambda ( 2 \pi \abs{w}) 
    = 
        2 \sin (\pi \abs{w})
    .
    \end{align*}
    The above is an equality if and only if $\abs{f^\prime(u)}=2\pi$ for all $u\in \Domain$ and $\Image(f)$ is a great circle.
\end{proof}
In order to prove critical estimates, we have to strengthen the before lemma to functions $[0,1]\rightarrow \mathbb{S}^1$ with monotonic angle function.
This can be interpreted as a $BV$ estimate.
\begin{lemma}\label{lem: max | Delta f| on S^1}
    Let $f:\Domain\rightarrow \mathbb{S}^1$ be injective and $w \in \intervaloo{0,1}$ be arbitrary.
    If one can parametrize $f$ by a monotonic angle function $\theta_f:[0,1]\rightarrow \R$, such that $\abs{\theta_f} \leq 2\pi$, then 
    \[
    	\int _{\Domain} 
    		\abs{ f(u+w) - f(u)} 
	\dd u 
	\leq 
	2 \,  \sin ( \pi w)
    .
    \]
    Furthermore, equality holds iff f is parametrized by constant speed $2\pi$.
\end{lemma}

\begin{proof}
    Let $f, w$ be given.
    We may assume, that the angle function $\theta_f:[0,1]\rightarrow [0,2\pi]$ is monotonically increasing, hence it is in $BV([0,1], \R)\cap \{ \norm{\theta}_{\Lebesgue[\infty]}\leq 2\pi\}$.
    From $f:\Domain \rightarrow \mathbb{S}^1$, we deduce that $\theta(1)=2\pi$.
    \\
    By~\cite[Theorem 5.3]{EvansGariepyMeasureTheory}, there exists a sequence $(\theta_k)_{k\in \N} \subset W^{1,1}$ such that $\theta_k \rightarrow \theta_f$ in $L^1$.
    Furthermore, by using a non-negative mollifier, we may assume that $\theta_k$ are monotonic and bounded from above by $2\pi$.
    We now define $f_k(t)= (\cos(\theta_k(t)), \sin(\theta_k(t)))$.
    Then $f_k \in L^\infty$ and $f_k \rightarrow f$ pointwise almost everywhere.
    A short computation now shows that
    \[
        \norm{f_k'}_{L^1}
        =
        \int_0^1 
            \abs{\theta_k'(t)}
        \dd t
        \leq 2\pi
        .
    \] 
    Hence, $f_k\in W^{1,1}$ with $\PathLength(f_k)\leq 2\pi$.
    Since $\dist[\mathbb{S}^n] (g(u+w),g(u))\leq 2 \pi$ for any $g: \Domain \rightarrow\mathbb{S}^n$ and $\dist[\mathbb{S}^n](\cdot, \cdot)$ is continuous, we can combine dominated convergence and \aref{lem:WirtingerMapsOnSphere} in order to compute 
    \begin{align*}
    \int _{\Domain} 
		\abs{ f(u+w) - f(u)} 
	\dd u
   	&\leq
    	\lambda \left( 
		\int _{\Domain} 
			\dist[\mathbb{S}^n] (f(u+w),f(u)) 
		\dd u 
	\right)
   	\\
    &=
    	\lim_{k \to \infty}  
		\lambda \left( 
			\int _{\Domain} 
				\dist[\mathbb{S}^n] (f_k(u+w),f_k(u)) 
			\dd u 
		\right)
    \leq
   	 \lim_{k \to \infty} 2  \sin ( \pi w) 
   	 \\
   	&=
   	2 \,  \sin ( \pi w)
    .
    \end{align*}
    Here the first inequality is an equality iff $\theta_f$ is parametrized by constant speed $2\pi$, hence if $f$ is parametrized by constant speed.
\end{proof}
The last two inequalities we need to prove are general consequences of the typical Hölder inequality and Young inequality. 
\begin{lemma}\label{lem:RevHoelder}
    Let $\Omega\subset \R^n$ be measurable and $r>1$.
    Furthermore, let $f,g:\Omega \rightarrow\R$ be measurable functions with $g(x) \neq 0$ for almost all $x\in \Omega$.
    Then
    \[
        \int_\Omega 
            \abs{f g}
        \dd x
        \geq
            (
                \int_\Omega 
                    \abs{f}^{1/r}
                \dd x
            )^r
            (
                \int_{\Omega}
                    \abs{g}^{-\frac{1}{r-1}}
                \dd x
            )^{-(r-1)}
    ,
    \]
    where equality holds, iff there exists an $\alpha \geq 0$ such that $|f|= \alpha |g|^{\frac{-r}{r-1}}$.
\end{lemma}
\begin{proof}
    Let $f,g,r$ be given as above.
    We denote $q=r'= \frac{r}{r-1}$ and compute the following.
    \begin{align*}
        &\int_\Omega
            \abs{f}^{1/r}
        \dd x
        =
        \int_\Omega
            \abs{ fg}^{1/r}
            \abs{g}^{-1/r}
        \dd x
        \leq
        (
            \int_\Omega
                \abs{fg}
            \dd x
        )^{1/r}
        (
            \int_\Omega
                \abs{g}^{q \frac{-1}{r}}
        )^{1/q}
    \end{align*}
    Taking both sides to power $r$ and rearranging the terms, one concludes the theorem.
\end{proof}

\begin{lemma}\label{lem:RevYoung}
	Let $p\in \intervaloo{0,1}, q\in \intervaloo{1,\infty}$ such that $\frac{1}{p}-\frac{1}{q}=1$.
	Then, for all $a,b>0$, one has
	\[
	ab\geq \frac{a^p}{p}- \frac{b^q}{q}
	\]
\end{lemma}
\begin{proof}
	The claim follows by applying the Young-inequality, with $u=1/p, v=q/p$ onto $x=(ab)^p$, $y=b^{-p}$ and rearranging the terms.
\end{proof}
For our methods, both inequalities can be applied similarly, producing the same result.
In particular, the parameter range in which $\TP^{(p,q)}$ can be shown to be basic, does not depend on the choice of the above two inequalities.
We will use the one, that appears to be more straightforward in the given situation.
Readers that are concerned with lower bounds for similar energies, might use the counterpart.
For example, if one minimizes geometric functionals with an additive pertubation, \aref{lem:RevYoung} might be better suited than \aref{lem:RevHoelder}.

\section{An approximate Fenchel-type Theorem}\label{sec:WirtingerAndLengthBounds}
In this chapter, we will closely investigate the path of $\GaussMap_\gamma$ on $\mathbb{S}^{n-1}$ and derive appropriate bounds.
Additionally, we provide some characterizations of convex curves in space.
Since the class of \textit{convex curves} will be of special interest, we give a definition of said class.
\begin{definition}
    We call an injective, closed space curve $\Domain \rightarrow \mathbb{R}^n$ \textit{convex}, if its image is the boundary of a convex (thus planar) set in $\mathbb{R}^n$ interpreted as a $2$-dimensional manifold.
\end{definition}
\noindent
We now investigate the length of paths of $\GaussMap_\gamma$ if one of the variables is fixed.
For $u\in \Domain$ and $\gamma \in W^{1,1}_{\mathrm{i,r}}(\Domain, \R^n)$, we define
\begin{equation}\label{eq:def eta}
    \eta_{\gamma,u}: \intervaloo{0,1}
     \rightarrow \mathbb{S}^{n-1}
    ,\;
    \eta_{\gamma,u}(w)
    \ceq 
    \GaussMap_\gamma(u,w)
    =
    \frac{\gamma(u+w)-\gamma(u)}{\abs{\gamma(u+w)-\gamma(u)}}
    .
\end{equation}
We choose $\intervaloo{0,1}$ as domain, because $\eta\in C((0,1))$. 
Moreover, if $\gamma '(u)$ exists, we observe that
\begin{equation}\label{eq:ComputationLimitEndpointEta}
	\lim_{w\rightarrow 0}\eta_{\gamma,u}(w)=\frac{\gamma '(u)}{\abs{\gamma '(u)}} \quad \text{and} \quad 
	\lim_{w\rightarrow 1}\eta_{\gamma,u}(w)=-\frac{\gamma '(u)}{\abs{\gamma '(u)}}
    .
\end{equation}
Furthermore, if $\gamma$ is convex, then for every $u \in \Domain$ 
\[
    \lim_{w\rightarrow 1}\eta_{\gamma,u}(w)
        \text{ and }
    \lim_{w\rightarrow 0}\eta_{\gamma,u}(w)
    \text{ exist.}
\] 
In both cases, $\eta$ may be interpreted as a continuous function $\intervalcc{0,1}\rightarrow \mathbb{S}^{n-1}$.
Thus, the second lower bound in \aref{thm: Approximate Fenchel Theorem} is a direct consequence of the fact that $\eta_{\gamma,u}$ continuously connects two antipodal points for almost all $u$. 
On the other hand, characterizing convex curves as the only extremizers of said bound, is a bit more involved. 
We dedicate the following part to that purpose.
The following lemma relates the shape of curves in space to the shape of $\eta_{\gamma,u}$ on $\mathbb{S}^{n-1}$.
\begin{lemma}\label{lem:Eta convex to great half circle}
    Let $\gamma \in W^{1,1}_{\mathrm{i,r}}$. 
    Then $\gamma$ is convex, iff for almost every 
    $u \in \Domain$: 
    $\Image (\eta _ {\gamma ,u})$ is contained in half of a great circle. 
\end{lemma}
\begin{proof}
    Suppose $\gamma$ is convex and let $u\in \Domain$ be arbitrary. 
    We may assume that $\gamma (u)=0$.
    Let $E \subset \R^n$ denote the plane, with $\gamma(\Domain)\subset E$. 
    In particular, $r \,\gamma(u+w)\in E$ for any $w$ and any $r \neq 0$. 
    This implies, that $\eta _{\gamma,u} ((0,1)) \subset E$ and by continuity of $\eta _{\gamma,u}$, we conclude $\eta _{\gamma,u} ([0,1]) \subset E$. 
    \\
    Convexity of $\gamma$ implies that $\Image(\gamma)$ is located on one side of the supporting plane at $\gamma(u)$ within $E$. 
    Since $\gamma(u)=0$, the same then holds for $\Image(\eta _{\gamma,u} )$.
    Thus, $\Image(\eta _{\gamma,u} )$ is a part of a half of great circle on $\mathbb{S}^{n-1}$. 
    \\
    Now assume that $\gamma\in W^{1,1}_{\mathrm{i,r}}$ is not convex. 
    If $\gamma$ is not planar, neither is $\eta_{\gamma, u}$ for any $u$.
    So in particular $\eta_{\gamma, u}$ is not contained in half of a great circle.
    Hence, we now assume that $\gamma$ is planar but not convex. 
    Without loss of generality, we assume that $\gamma: \Domain \rightarrow \R^2$.
    Since $\gamma$ is not convex, there is a $u\in \Domain$, such that there is no supporting hyperplane at $u$.
    This implies, that for every $v\in \mathbb{S}^1$, there are $w_1,w_2$ such that
    \[
    \inner{\gamma(u+w_1)-\gamma(u), v}>0
    \text{ and }
    \inner{\gamma(u+w_2)-\gamma(u), v}<0
    .
    \]
    We denote $y_i(u,v)=u+w_i$.
    Since $(u,y,n)\mapsto \inner{\gamma(y)-\gamma(u),n}$ is continuous, there is a $\delta=\delta(v)>0$, sufficiently small, such that
    \[
    \inner{\gamma(y_2(u,v))-\gamma(z), n}
    <0
    <
    \inner{\gamma(y_1(u,v))-\gamma(z), n}
    \text{ for all }
    z\in B_{\delta}(u),n\in B_\delta(v)\cap \mathbb{S}^{n-1}
    .
    \]
    Furthermore, $\{B_{\delta(v)}(v), v\in \mathbb{S}^{n-1}\}$ covers $\mathbb{S}^{n-1}$, hence we can extract a finite subcover $\{ B_{\delta(v_i)}(v_i), i=1,...,N\}$.
    Let $\delta=\min_{i=1,...,N}\delta(v_i)$.
    We claim, that for all $z\in B_\delta(u)$ $\Image(\eta_{\gamma,z})$ is not contained in a closed hemisphere.
    This is true, since for $z\in B_\delta(u), n\in \mathbb{S}^{n-1}$, there is an $i$ such that
    $n\in B_{\delta}(v_i)$.
    Therefore,
    \[
    \inner{\gamma(y_2(u,v_i))-\gamma(z), n}
    <0
    <
    \inner{\gamma(y_1(u,v_i))-\gamma(z), n}
    .
    \]
    Hence, $\Image(\eta_{\gamma,z})$ is not contained in a closed hemisphere, so in particular it is not contained in half of a great circle.
\end{proof}
\begin{remark}\label{rmk:AngleEtaMonotonic}
    A consequence of the above is, that if  $\Image(\eta_{\gamma,u})$ is located on a half of a great circle for almost all $u\in \Domain$,
    then $\dist[\mathbb{S}^{n-1}][\eta_{\gamma,u}(0)][\eta_{\gamma,u}(w)]$ is monotonically increasing:\\
    By the above Lemma, we immediately conclude from the assumption that $\gamma$ is convex.
        Suppose now by contradiction, that there are $0\leq w_1<w_2<1$ such that
    \[
       \dist[\mathbb{S}^{n-1}]\pars{\eta_{\gamma,u}(0),\eta_{\gamma,u}(w_1)}
        >
        \dist[\mathbb{S}^{n-1}]\pars{\eta_{\gamma,u}(0),\eta_{\gamma,u}(w_2)}
    .
    \]
    Since we already know that $\Image(\eta_{\gamma,u})$ is contained in a half of a great circle, connecting $\eta_{\gamma,u}(0)$ and $\eta_{\gamma, u}(1)$, we may assume that there is $w_3$ with $w_2<w_3 <1$ such that $\eta_{\gamma,u}(w_1)= \eta_{\gamma,u}(w_3)$.
    Because $\gamma'(z)$ exists for almost all $z$ and due to continuity of $\eta_{\gamma,u}$, we may assume that $\gamma'(w_i)$ exists for all $i=1,2,3$.
    By definition, $\Image(\gamma)$ is the boundary of a convex, planar set $\Gamma$.\\
    Since $\gamma(u),\gamma(u+w_1), \gamma(u+w_3)\in \Image(\gamma)=\partial \Gamma$, the line $\gamma(u)+\text{span}(\eta_{\gamma,u}(w_1))$ intersects $\partial \Gamma$ at three points.
    Therefore, we know that $\gamma(u+[w_1,w_3])$ is a line segment, contained in $\text{span}(\eta_{\gamma,u}(w_1))$.
    This implies that $\gamma(u+w_2)$ is contained in said line segment and therefore, we have that $\eta_{\gamma,u}(w_1)= \eta_{\gamma, u}(w_2)$, contradicting the assumption.
    Note that this also implies, that $w\mapsto\Angle{\eta_{\gamma,u}(0)}{\eta_{\gamma,u}(w)}$ is monotonic.
\end{remark}
We are now ready to prove the second part of \aref{thm: Approximate Fenchel Theorem}.

\begin{lemma}\label{lem:LengthOfGaussMapinW}
    Let $\gamma\in W^{1,1}_{\mathrm{i,r}}(\Domain, \AmbSpace)$.
    Then for almost every $u\in \Domain$, the following inequality holds true
    \[
        \PathLength(w\mapsto \GaussMap_\gamma(u,w))\geq \pi
    \]
    with equality for almost every $ u \in \Domain$ iff $\gamma$ is a planar, convex curve. 
    Furthermore, the equality holds for all $u\in \Domain$, iff $\gamma$ is a convex $C^1$-curve.
\end{lemma}
\begin{proof}
	Let $\gamma \in W^{1,1}_{\mathrm{i,r}}(\Domain, \AmbSpace)$. 
    If $\PathLength(w\mapsto \GaussMap_\gamma(u,w))=\pi$ for $u \in U \subset \Domain$ with $\lambda(U)=1$, then the subset $\{u \in U \, | \, \gamma '(u) \text{ exists} \}$ also has measure $1$.
	Now fix $u \in \Domain$ such that $\gamma '(u)$ exists.
	Then $\eta_{\gamma ,u}$ continuously connects the two antipodal points $\gamma '(u)/\abs{\gamma '(u)}$ and $-\gamma '(u)/\abs{\gamma '(u)}$ on $\mathbb{S} ^n$. 
    Therefore,
	\[
	    \PathLength(w\mapsto \GaussMap_\gamma(u,w))
        = 
        \PathLength(\eta_{\gamma ,u}) \geq \pi
        ,
	\]
	with equality iff $\Image(\eta_{\gamma,u})$ is a great half circle and $\dist_{\mathbb{S}^{n-1}}(\eta (\cdot), \gamma'(u))$ is monotonically increasing.
	By \aref{lem:Eta convex to great half circle} and \aref{rmk:AngleEtaMonotonic} we conclude that, $\PathLength(w\mapsto \GaussMap_\gamma(u,w))=\pi$ for almost every $u \in \Domain$, iff $\gamma$ is convex.
	\\ 
    \noindent
    If equality holds for all $u$, we first conclude that $\gamma$ is convex.
    Furthermore, \aref{lem:Eta convex to great half circle} and \aref{rmk:AngleEtaMonotonic} imply $\eta_{\gamma,u}(1)= -\eta_{\gamma,u}(0)$.
    Hence, the one-sided differentials at $u$ coincide. 
    Using convexity, we conclude that $\gamma$ is $C^1$.
    If $\gamma$ is convex and $C^1$, the claim follows as a combination of \aref{eq:ComputationLimitEndpointEta}, \aref{lem:Eta convex to great half circle} and \aref{rmk:AngleEtaMonotonic}.
\end{proof}
We now prove the first part of \aref{thm: Approximate Fenchel Theorem}.
For given $\gamma\in W^{1,1}_{\mathrm{i,r}}, w\in (0,1)$ we define
\begin{equation}\label{eq:def rho}
	\rho_{\gamma,w}:\Domain \rightarrow \mathbb{S}^{n-1}
	,\;
	\rho_{\gamma,w}(u)
	\ceq 
	\GaussMap_\gamma(u,w)
	=
	\frac{\gamma(u+w)-\gamma(u)}{\abs{\gamma(u+w)-\gamma(u)}}
	.
\end{equation}
Note that, in contrast to $\eta_{\gamma,u}$, the function $\rho_{\gamma,w}$ does not necessarily connect two antipodal points.
Instead, $\Image(\rho_{\gamma,w})$ is a closed curve on $\mathbb{S}^n$. 
Therefore, the lower bound on $\PathLength(\rho_{\gamma,w})$ is more involved. 
In order to obtain said lower bound, we prove that $\Image (\rho_{\gamma,w})$ can not lie in an open hemisphere, mimicking the proof of Fenchel's theorem. 
Before we prove that, we state the following lemma as an analog to \aref{lem:Eta convex to great half circle}. 
It poses to be essential for characterizing convex curves as minimizers.\\
In order to do so, we use the following convention.
For $v\in \mathbb{S}^{n-1}$ and $f\in C([0,T],\mathbb{S}^{n-1})$, the function $\theta(x)\ceq \Angle{v}{f(x)}$ denotes the continuous lift of the angle function.
\begin{lemma}\label{lem:RhoGreatCircleMonotonicAngleGammaConvex}
    Let $\gamma\in W^{1,1}_{\mathrm{i,r}}$.
    Then $\gamma$ is convex iff
    $\Image(\rho_{\gamma,w})$ is a great circle for all $w\in (0,1)$ and $\rho_{\gamma,w}$ is a parametrization with monotonic angle function 
    $\theta \ceq \theta _{\gamma , w}:[0,1] \rightarrow\R, u\mapsto \Angle{\rho_{\gamma,w}(0)}{\rho_{\gamma,w}(u)}$ satisfying $\theta(1)\in \{\pm 2 \pi \}$.
\end{lemma}
\begin{proof}
    If the statement holds true for $\gamma(\cdot)$, then it also holds true for $\gamma(0-\cdot)$.
    Hence, we may assume that the parametrization of the angle is always monotonically increasing.
    Let $\gamma\in W^{1,1}_{\mathrm{i,r}}$ be convex.
    Then $\gamma$ is planar. 
    In particular, this implies that $\Image(\rho_{\gamma,w})$ is contained in a great circle.
    From now on, we assume that $\Image(\gamma)\subset \R^2$.
    Applying a rotation, we may assume that $\frac{\gamma'(0)}{\abs{\gamma'(0)}}=e_1$.
    Recall that $\theta_{\gamma,w}$ is the continuous angle function, measured with respect to $\rho_{\gamma,w}(0)$.
    One computes that
    \[
        \frac{\dd}{\dd u} \theta_{\gamma,w}(u)
        =
        \det(\rho_{\gamma,w}(u), \rho_{\gamma,w}'(u))
        =
        \frac{
            \det(\rho_{\gamma,w}(u), \gamma'(u+w)-\gamma'(u))
        }
        {\abs{\Delta_{u+w,u}\gamma}} 
    \text{ almost everywhere.}
    \]
    It now suffices to show that $\det(\rho_{\gamma,w}(u), \gamma'(u+w)-\gamma'(u))\geq 0$ for almost every $u\in[0,1]$.
    Note that 
    \[
        \det(\rho_{\gamma,w}(u), \gamma'(u+w)-\gamma'(u))
        =
        \sin(
            \Angle{\frac{\gamma'(u+w)}{\abs{\gamma'(u+w)}}}{\rho_{\gamma,w}(u)}
        )
        -
        \sin(
            \Angle{\frac{\gamma'(u)}{\abs{\gamma'(u)}}}{\rho_{\gamma,w}(u)}
        )
    ,
    \]
    Since $\gamma$ is convex, we know that 
    \[
            -\pi 
        \leq 
            \Angle{\frac{\gamma'(u)}{\abs{\gamma'(u)}}}{\rho_{\gamma,w}(u)}
        \leq 
            0 
        \leq
            \Angle{\frac{\gamma'(u+w)}{\abs{\gamma'(u+w)}}}{\rho_{\gamma,w}(u)}
        \leq
            \pi
        \;(
            \text{mod }
            2\pi
        )
        \;
    .
    \]
    Hence, $\theta_{\gamma,w}'(u) \geq 0$ for almost all $u\in [0,1]$.
    Furthermore, $\rho_{\gamma,w}(0)=\rho_{\gamma,w}(1)$ implies that $\theta_{\gamma,w}(1)=2\pi$.
    \\
    Now suppose that $\Image(\rho_{\gamma,w})$ is a great circle and $\theta_{\gamma,w}$ is monotonically increasing for all $w$.
    We may assume that $\gamma(0)=0$ and that for a fixed, irrational $w\in [0,1]$ $\Image(\rho_{\gamma,w})=\mathbb{S}^{n-1} \cap \R^2$.
    Let $X=\{ w \N \mod 1\}$.
    A standard argument implies that $X\subset \Domain$ is dense.
    Since $\rho_{\gamma, w}(u)\in \R^2$ for all $u\in X$, by a telescopic sum, we conclude that $\gamma(w)\in \R^2$ for all $w\in X$.
    Due to the continuity of $\gamma$, we conclude that $\gamma$ is a planar curve.\\
    Moreover, a short computation shows that
    \[
        \lim_{w\searrow 0}\rho_{\gamma, w}(u)
        =
        \frac{\gamma'(u)}{\abs{\gamma'(u)}}
    \]
    for almost all $u\in [0,1]$. 
    Let $\Angle{e_1}{\cdot}:\mathbb{S}^1\rightarrow \R$ be the continuous lift angle function, measuring with respect to $e_1$.
    Thus, by continuity of $\Angle{e_1}{\cdot}$ on $\mathbb{S}^1$ and $\theta_{\gamma,w}(u)=\Angle{\rho_{\gamma,w}(0)}{e_1} + \Angle{e_1}{\rho_{\gamma,w}(u)}$,
    \[
    	\lim_{w \searrow 0} \theta_{\gamma, w} 
        = 
        \Angle{e_1}{\frac{\gamma'(\cdot)}{\abs{\gamma'(\cdot)}}} 
        \quad \text{almost everywhere.}
    \]
    Therefore, monotonicity of $\theta_{\gamma,w}$ for all $w$ implies monotonicity of $\Angle{e_1}{\frac{\gamma'(\cdot)}{\abs{\gamma'(\cdot)}}}$.
    A standard argument, relying on the supporting hyperplane theorem, now implies that $\gamma$ is convex.
\end{proof}
We are now equipped to prove the central statement of the approximate Fenchel theorem.
\begin{lemma}\label{lem:LengthOfGaussMapinU}
    Let $\gamma \in W_{\mathrm{i,r}}^{1,1}(\Domain, \R^n)$. 
    Then for all $w\in (0,1)$, the following inequality holds true
    \[
        \PathLength(u\mapsto \GaussMap_\gamma(u,w))
        \geq
        2\pi
    \]
    with equality for all $w$ iff $\gamma$ is a planar, convex curve.
\end{lemma}
\begin{proof}
    Let $\gamma \in W_{\mathrm{i,r}}^{1,1}$ and $w \in (0,1)$. 
    Again we observe that $\rho_{\gamma,w}$ is a closed $W^{1,1}$-curve on $\mathbb{S}^{n-1}$.
    We claim that $\Image(\rho_{\gamma,w})$ is not contained in one open hemisphere.
    The proof of this property mimics the proof of Fenchel's theorem about the total curvature of curves on spheres.
    If the image is contained in an open hemisphere, there exists an $n\in \mathbb{S}^{n-1}$ such that
    \[
        \inner{\rho_{\gamma,w}(u), n}= \inner{\GaussMap_\gamma(u,w), n} >0
        \text{ for all }
        u\in \Domain.
    \]
    This implies that $\inner{\gamma(u+w)-\gamma(u), n} >0 $ for all $u\in \Domain$.
    Integrating now yields that
    \begin{align*}
        0&<
            \int_\Domain
                \inner{\gamma(u+w)-\gamma(u), n} 
            \dd u   
        =
                \inner{ \int_\Domain \gamma(u+w) \dd u,n}
            -  
                \inner{ \int_\Domain \gamma(u) \dd u,n}
        =
        0
        .
    \end{align*}
    Hence, $\rho_{\gamma,w}$ cannot be contained in an open hemisphere.\\
    A standard argument shows that the length of any curve $c:\Domain \rightarrow \mathbb{S}^n$, whose image is not contained in an open hemisphere, is at least $2 \pi$. 
    For completeness, we carry out the proof.
    If $n=2$, there is nothing left to show. 
    Therefore, we may assume that $n\geq 3$.
    Choose $T\in (0,1)$, such that the arcs $c \vert_{[0,T]}$ and $c\vert_{[T,1]}$ have the same length.
    Without loss of generality, we may assume that $c(0)=(a,0,b), c(T)=(-a,0,b)$, where $a\in \R, b\in \R^{n-2}$.
    Since $\Image(c)$ is not contained in an open hemisphere,
    we know that either $c \vert_{[0,T]}$ or $c\vert_{[T,1]}$ crosses the equator $\{x_1=0\}$.
    We assume that $c \vert_{[0,T]} \cap \{ x_1=0\} \neq \emptyset$ and choose $p \in \rho \vert_{[0,T]} \cap \{ x_1=0\}$.
    Let $R$ be a rotation around the $x_2$-axis with $R((a,0,b))=(-a,0,b)$.
    Therefore, $R(c(T))=c(0)$.
    We now define
    \[
        \Tilde{c}(t)
        =
        \begin{cases}
            c(t), t\in [0,T],\\
            Rc( (t-T) \frac{T-1}{1-T}), t\in [T,1]
        \end{cases}
        .
    \]
    $\Tilde{c}$ is a closed loop in $\mathbb{S}^n$ with
    \[
        \PathLength(\Tilde{c})= 2 \PathLength( c\vert_{[0,T]})
        =
        \PathLength(c)
        .
    \]
    Since $\Tilde{c}$ contains $p$ and $-p$, we know that $\PathLength(\Tilde{c}) \geq 2\pi$ with equality iff $\Tilde{c}$ is a parametrization of a great circle with monotonic angle function.
    This is the case, iff the same holds true for $c$. 
    Using \aref{lem:RhoGreatCircleMonotonicAngleGammaConvex}, we conclude that
    \[
        \PathLength(\rho_{\gamma ,w}) \geq 2\pi \quad \forall w \in \intervaloo{0,1}
    \]  
    with equality for all $w$ iff $\gamma$ is a planar and convex curve.
\end{proof}
\begin{remark}
    Since both statements, \aref{lem:LengthOfGaussMapinU} and \aref{lem:LengthOfGaussMapinW}, are estimates of the length,
    it does not matter whether one chooses the standard differentiation $\partial_u$ (respectively $\partial_w$) and the standard integration $\dd u$ (respectively $\dd w$)
    or the geometric differentiation by arc length $\partial_{s(u)}=\frac{1}{\abs{\gamma'(u)}}\partial_u$ (respectively $\partial_{s(w)}= \frac{1}{\abs{\gamma'(u+w)}}\partial_w$) and the integration w.r.t.\ arc length.
\end{remark}

%% file: RelatedFunctionals.tex
\section{Minimizing functionals depending on the direction of the secant}\label{sec:RelFunct}

In this section we minimize three types of functionals, naturally arising as generalized minorants of the generalized tangent-point energies.
The application to the specific minorants will be discussed in \aref{sec:MinMostTPE}. 
All estimates are consequences of the bounds before.
In order to make the bounds to as many functionals as possible, the functionals $\Funct_1$, $\Funct_2$ and $\Funct_3$ are defined as general as the used methods allow. 
Similar terms as the tangent point energies occur for example in fractional perimeter functionals or the fractional Willmore energies.\\
Furthermore, $\Funct_1$ is deliberately chosen, not to be invariant under reparametrizations, because the minimization implies parametrization by arc length.
The corresponding, invariant functional, can be recovered by simply restricting the domain to curves parametrized by arc length, because the bounds from \aref{sec:WirtingerAndLengthBounds} holds in $W^{1,1}_{\mathrm{i,r}}(\Domain, \R^n)$.
The downside of this general treatment is a more involved characterization of minimizing curves, which is outsourced to the appendix (see \aref{lem: constant derivative of GaussMap} and \aref{rmk: constant derivative of GaussMap}).
Moreover, for each functional $\Funct_1$, $\Funct_2$ and $\Funct_3$ its sharp lower bound is presented as a corollary of the corresponding bound, that still holds true, if one fixes a suitable variable in their definition.
This can be interpreted as minimization in $L^\infty$.\\
As an immediate consequence of \aref{lem:LengthOfGaussMapinW} and \aref{lem:LengthOfGaussMapinU}, we obtain the following.

\begin{lemma}
	\label{lem: convex f wu bound}
	Let $f: \intervalco{0,\infty}\rightarrow \intervalco{0,\infty}$ be convex and monotonically increasing and $\gamma \in W^{1,1}_{\mathrm{i,r}}(\Domain, \mathbb{R}^n)$. Then 
	\begin{equation}
		\label{eq: convex f w bound}
		\int_{\Domain} 
		f(\abs{ \partial_w \GaussMap_\gamma})
		\dd w \geq f(\pi) \quad \text{for almost every } u \in \Domain,
	\end{equation}
	and
	\begin{equation}
		\label{eq: convex f u bound}
		\int_{\Domain} 
			f(\abs{ \partial_u \GaussMap_\gamma})
		\dd u 
		\geq f(2\pi) 
		\quad \text{for all } w \in (0,1),
	\end{equation}
	with equality for the circle.
	If $f$ is strictly convex and \aref{eq: convex f w bound} is an equality for almost every $u \in \Domain$ or \aref{eq: convex f u bound} is an equality for all $w \in (0,1)$, then $\Image (\gamma)$ is a circle, parametrized by constant speed.
\end{lemma}

\begin{proof}
	Let $f$ and $\gamma$ be given as above. Let $z \in \{u,w\}$ and $c_w=\pi, c_u=2 \pi$. Combining Jensen's inequality with each of \aref{lem:LengthOfGaussMapinW} and \aref{lem:LengthOfGaussMapinU} immediately yields
	\[
		\int_{\Domain} 
			f(\abs{ \partial_z \GaussMap_\gamma}) 
		\dd z
		\geq 
		f\pars*{	
			\int_{\Domain} \abs{ \partial_z \GaussMap_\gamma} \dd z 
		} 
		\geq f(c_z) 
	\]
	for almost all $u \in \Domain$, and for all $w \in (0,1)$, respectively. 
	If $\Image (\gamma)$ is a circle, parametrized by constant speed, then $\Image  (w \mapsto \GaussMap _\gamma (u,w))$ is a great half circle, parametrized with the constant speed $\abs{ \partial_w \GaussMap_\gamma}=\pi$, for all $u \in \Domain$.
	Furthermore, $\Image (w \mapsto \GaussMap _\gamma (u,w))$ is a great circle, parametrized with the constant speed $\abs{ \partial_u \GaussMap_\gamma}=2\pi$, for all $w \in (0,1)$.
	Hence, 
	\[
		\int_{\Domain} 
			f(\abs{ \partial_z \GaussMap_\gamma}) 
		\dd z
	=
		\int_{\Domain} 
			f(c_z) 
		\dd z 
	= 
		f(c_z)
	.
	\]
	Now let $f$ be strictly convex and assume that \aref{eq: convex f w bound} or \aref{eq: convex f u bound} is an equality for almost every $u \in \Domain$ or every $w \in (0,1)$, respectively. 
	Since Jensen's inequality is an equality, iff the integrand is constant almost everywhere, we conclude that $\abs{ \partial_z \GaussMap_\gamma}=c_z$ for almost every $u \in \Domain$ and $w \in (0,1)$.
	By \aref{lem: constant derivative of GaussMap}, this implies that $\Image (\gamma)$ is a circle, parametrized by constant speed.
\end{proof}
If we consider the integral w.r.t.\ both variables, we obtain as a corollary the following theorem.

\begin{theorem}\label{thm:MinI3}
	Let $f: \intervalco{0,\infty}\rightarrow \intervalco{0,\infty}$ be convex and monotonically increasing. Let $z \in \{u,w\}$.
	Consider $\Funct_1:W^{1,1}_{\mathrm{i,r}}\rightarrow \R$ given by
	\[
		\Funct_1(\gamma)
	=
		\int_\Domain\int_\Domain
			f(\abs{ \partial_z \GaussMap_\gamma}) 
		\dd w \dd u
	.
	\]
	Then $\Funct_1 (\gamma) \geq f(c_z) $, where $c_w=\pi$ and $c_u=2\pi$, with equality for the circle.
	If $f$ is strictly convex, then equality holds, iff $\Image (\gamma)$ is a circle, parametrized by constant speed.
\end{theorem}
For the next bound, we have to restrict the domain to arc length parametrized curves, in order to apply \aref{lem: circle max of Delta gamma |^r}. 
On the other hand, in contrast to \aref{lem: convex f wu bound} we can therefore establish uniqueness of the circle as a minimizer for an arbitrary but fixed $w \in (0,1)$.

\begin{lemma}
	\label{lem:MinI4}
	Let $f,g:\intervalco{0,\infty} \rightarrow \intervalco{0,\infty}$ be strictly monotonically increasing. 
	Let $w \in (0,1)$ and $\gamma \in W^{1,\infty}_{\mathrm{i,a}}$. 
	Suppose that there exists $\theta>1$, such that $f^{1/\theta}$ is convex and $g^{1/(\theta-1)}$ is concave. 
	Then
	\[
	\int_\Domain
		f(\abs{\partial_u \GaussMap_\gamma})
		\frac{1}{g(\abs{\gamma(u+w)-\gamma(u)}^2)}
	\dd u
	\geq 
		f(2\pi) \, 
		\frac{1}{g\left(\pi^{-2} \, 
		\sin (\pi w)^2 \right)} \,
		,
	\]
	where equality holds iff $\Image (\gamma)$ is a circle. 
\end{lemma}

\begin{proof}
	Let $f,g, \theta, w$ be given as above and $\gamma \in W^{1,\infty}_{\mathrm{i,a}}$. 
	Applying \aref{lem:RevHoelder} to $\theta,f,1/g$ we conclude that
	\begin{align*}
		&\int_\Domain
			f(\abs{\partial_u \GaussMap_\gamma})
			\frac{1}{g(\abs{\gamma(u+w)-\gamma(u)}^2)}
		\dd u
	\\
	&\geq
		\left(
			\int_\Domain
				f(\abs{\partial_u \GaussMap_\gamma})^{\frac{1}{\theta}}
			\dd u
		\right)^\theta
		\left(
			\int_\Domain
				g(
					\abs{
						\gamma(u+w)-\gamma(u)
					}^2
				)^{\frac{1}{\theta-1}}
			\dd u
		\right)^{-(\theta-1)}
	\end{align*}
	Applying \aref{lem: convex f wu bound} to $f^{1/\theta}$ and \aref{lem: circle max of Delta gamma |^r} to $g^{1/(\theta -1)}$, noting that $x \mapsto x^{-(\theta -1)}$ is monotonically decreasing, we obtain the desired bound. 
	Uniqueness of the circle as a minimizer follows from  \aref{lem: circle max of Delta gamma |^r}.
\end{proof}
Integrating with respect to $w$ yields the following theorem as a corollary.	

\begin{theorem}\label{thm:MinI4}
    Let $f,g:\intervalco{0,\infty} \rightarrow \intervalco{0,\infty}$ be strictly monotonically increasing. 
	Suppose that there exists $\theta>1$, such that $f^{1/\theta}$ is convex and $g^{1/(\theta-1)}$ is concave. 
	Consider the functional
    $\Funct_2: W^{1,\infty}_{\mathrm{i,a}}\rightarrow \R$ given by
    \[
        \Funct_2(\gamma)
        =
        \int_\Domain
        \int_\Domain
            f(\abs{\partial_u \GaussMap_\gamma})
            \frac{1}{g(\abs{\gamma(u+w)-\gamma(u)}^2)}
        \dd u
        \dd w
    ,
    \]
    Then 
    \[
    \Funct _2 (\gamma) \geq f(2 \pi) \int_0^1 \frac{1}{g\left(\pi^{-2} \, \sin (\pi w)^2 \right)} \dd w
    \,,
    \]
	where equality holds, iff $\Image (\gamma)$ is a circle.
\end{theorem}
For the following lemma, we again do not need to restrict ourselves to arc length parametrized curves.

\begin{lemma}\label{lem:MinI5}
	Let $\varepsilon >0$, $g \in L^{2+\varepsilon}((-1,1), \intervalco{0, \infty})$. 
	Let $\gamma \in  W^{1,1}_{\mathrm{i,r}}$. Then for almost every $u \in \Domain$
	\[
	\int_\Domain
	g(\langle t_\gamma (u), \GaussMap_\gamma \rangle ) \, \abs{\partial_w \GaussMap_\gamma}
	\dd w 
	\geq 
	\pi \, \int_0^1 g(\cos(\pi w)) 
	\dd w \, ,
	\]
	where equality holds for almost every $u \in \Domain$, iff $\gamma$ is convex. Here $t_\gamma (u) \coloneq \gamma '(u) / \abs{\gamma '(u)}$.
\end{lemma}

\begin{proof}
	Let $g$ and $\gamma$ be given as above. We fix a $u \in \Domain$ such that $\gamma '(u)$ exists, by assumption it does not vanish. 
	Referring to \aref{eq:def eta} and abbreviating $\eta= \eta_{\gamma ,u}$, we have $\eta (0)=t_\gamma (u)$ and $\eta (1)=- t_\gamma (u)$.
	\\
	Let $K(x)\ceq \int_{x}^1 g(y) \, (1-y^2)^{-\frac{1}{2}} \dd y$ be a primitive of $x\mapsto - g(x) \, (1-x^2)^{-\frac{1}{2}}$.
	Since $g \in L^{2+\varepsilon}$ and $x \mapsto (1-x^2)^{-\frac{1}{2}} \in L^{(2+\varepsilon)^\ast}$, where $(2+\varepsilon)^\ast <2$ denotes the Hölder conjugate of $2+ \varepsilon$, the primitive is well-defined as a consequence of Hölder's inequality.
	We define $h(v)\ceq K(\inner{t_\gamma (u), v})$, $h: \mathbb{S}^{n-1} \rightarrow \R$.
	Then by substituting $y=\cos(\pi w)$ we compute
	\begin{align*}
		h(-t_\gamma (u))-h(t_\gamma (u))
	=&
		K(-1)-K(1)
		=
		\int_{-1}^1 
			g(y) \, 
			(1-y^2)^{-\frac{1}{2}} 
		\dd y
	\\
	=&
		\int_0^1 
			g(\cos(\pi w)) 
			(1-\cos( \pi w)^2)^{-\frac{1}{2}}
			\, \sin( \pi w) \, \pi
		\dd w
	\\
	=&
	\pi \, 
		\int_0^1 
			g(\cos(\pi w)) 
		\dd w
	\end{align*}
	On the other hand one computes 
	$
		Dh(v) w
		=
		- g(\inner{t_\gamma (u), v}) \, 
		(1- \inner{t_\gamma (u), v}^2)^{-\frac{1}{2}}
		\inner{t_\gamma (u), w}	
	$.	
	Hence,
	\begin{align*}
		h(-t_\gamma (u))-h(t_\gamma (u))
		=&
		\int_{0}^{1}
		Dh(\eta)(\eta')
		\dd w
		\\
		=&
		\int_0^1 
		g(\inner{t_\gamma (u), \eta (w)}) \,
		\abs{P_{\gamma'(u)}^\perp \eta (w)}^{-1}
		\inner{- t_\gamma (u), \eta'(w)
		}
		\dd w
		\\
		=&
		\int_0^1 
		g(\inner{t_\gamma (u), \eta (w)}) \,
		\abs{P_{\gamma'(u)}^\perp \eta (w)}^{-1}
		\inner{-P_{\eta (w)} ^\perp t_\gamma (u), \eta'(w)
		}
		\dd w
		\\
		\leq&
		\int_0^1 
		g(\inner{t_\gamma (u), \eta(w)}) \,
		\abs{\eta'(w)}
		\dd w
		.
	\end{align*}
	In the above, equality holds, iff $-P_{\eta}\gamma'(u)$ is a positive multiple of $ \eta '$ for almost every $w$.
	This is only the case, if $\Image (\eta)$ is a great half circle and $\dist[\mathbb{S}^{n-1}][\eta(0)][\eta(w)]$ is monotonically increasing in $w$.
	Invoking \aref{lem:Eta convex to great half circle} and \aref{rmk:AngleEtaMonotonic}, equality holds for a.e. $u \in \Domain$, iff $\gamma$ is convex.
\end{proof}
By integrating with respect to $u$, we obtain as a corollary the following theorem.

\begin{theorem}\label{thm:MinI5}
	Let $\varepsilon >0$, $g \in L^{2+\varepsilon}((-1,1), \intervalco{0, \infty})$. Consider the functional  $\Funct_3: W^{1,1}_{\mathrm{i,r}}\rightarrow \R$ given by
	\[
	\Funct_3(\gamma)
	=
	\int_\Domain
	\int_\Domain
	g(\langle t_\gamma (u), \GaussMap_\gamma \rangle ) \, \abs{\partial_w \GaussMap_\gamma}
	\dd w \dd u
	\]
	Then 
	\[
	\Funct_3 (\gamma) \geq \pi \, \int_0^1 g(\cos(\pi w)) 
	\dd w \, ,
	\]
	where equality holds, iff $\gamma$ is convex. 
	Here $t_\gamma (u)  \coloneq \gamma '(u) / \abs{\gamma '(u)}$.
\end{theorem}

%% file: MinTPE.tex
\section{Minimizing most tangent-point energies}\label{sec:MinMostTPE}

The goal of this chapter is proving that $\TPE^{(p,q)}$ from \aref{eq:DefTPpq} is uniquely minimized by round circles for wide ranges of $p$ and $q$.
In order to do so, we will use the relation between the unit secant $\GaussMap _{\gamma}$ and the integrand of $\TPE^{(p,q)}$ (see \aref{eq:DerivGaussMapW}, \aref{eq:DerivGaussMapU}). 
More concretely, the strategy is as follows. 
For fixed $p$ and $q$ -upon restriction to arc length parametrized curves- we will first present a suitable minorant for $\TPE^{(p,q)}$, touching it at the circle, 
i.e. a functional $H$, satisfying
\[
	 H \leq \TPE ^{(p,q)} \text{ on } W^{1,\infty}_{\mathrm{i,a}}
	 \text{ and }
	 H(\gamma _C)=\TP^{(p,q)}(\gamma _C)
	 .
\]
In the above, $\gamma_C$ denotes an arbitrary circle with circumference $1$, parametrized by arc length.\\
The minorant $H$ will be a special case of $\Funct_1$, if $p=2q$, of $\Funct_2$, if $p>2q$ and of $\Funct_3$, if $p=q+1$.
Invoking our findings from \aref{sec:RelFunct}, we then conclude that such $H$ are minimnized by circles. 
By the standard argument
\begin{equation}
	\label{eq: minorant argument}
	\TPE ^{(p,q)} (\gamma _C) = H(\gamma _C) \leq 
	H(\gamma ) \leq \TPE ^{(p,q)} (\gamma ) \quad \forall \gamma \in W^{1, \infty}_{\mathrm{i,a}}
	,
\end{equation} 
$\TPE^{(p,q)}$ is then also minimized by the circle among arc length parametrized curves. Lastly, since the tangent-point energies are invariant under reparametrizations and $\TP^{(p,q)}(\lambda \gamma)= \lambda^{q+2-p}\TP^{(p,q)}(\gamma)$, this result extends to all regular curves with fixed length. This step will be ommited in the following theorems.
For $p=q+1$ some additional effort is needed, in order to show uniqueness of the circle as a global minimizer, as $\Funct_3$ is minimized by any convex curve. \\
We start by reestablishing the already known results for $p=2q$.

\subsection{The geometric case $p=2q$}
\label{subsection:The geometric case $p=2q$}
This choice of parameters is called "geometric", since the generalized tangent-point energy simplifies to
\[
    \TPE^{(2q,q)}(\gamma)
    =
    \iint_{\Domain \times \Domain}
    \frac{2^q}{\text{r}_\text{TP}(\gamma)(x,y)^q}
    \abs{\gamma'(x)}
    \abs{\gamma'(y)}
    \dd(x,y) 
    = 2^q \, \TP ^{(q)}(\gamma)
    .
\]
It is known that the minimizers $\TPE^{(2q,q)}$ are the circles for $q\geq1$.
The proof occured in \cite{volkmann1}, as a byproduct of a monotonicity formula.
It is also pointed out by the author, that Blatt found a more straightforward way to prove \aref{thm: circle min p=2q}.
Using our techniques, we were able to slightly improve the required regularity.
\begin{theorem}
\label{thm: circle min p=2q}
    Let $q \geq 1$. 
    Then $\TP^{(2q,q)}(\gamma) \geq \pi ^q$ for every $\gamma \in W_{\mathrm{i,a}}^{1,\infty}$.
    Furthermore, if $q>1$ equality holds iff $\Image( \gamma)$ is a circle and if $q=1$ equality holds iff $\gamma$ is convex.
\end{theorem}
The original proof relies on Blatt's observation that $\abs{\partial_w \GaussMap_\gamma}(u,w)=\pars*{2\text{r}_{\text{TP}}(\gamma)(u+w,u)}^{-1}$.
By the before theorems, this can be seen as an immediate consequence of \aref{thm:MinI3} with the choice $z=w$, thus relying on the length bound in \aref{lem:LengthOfGaussMapinW}. \\
Since $x\mapsto x^q$ is convex for $q\geq 1$ and strictly convex for $q>1$, we immediately conclude that
\[
    \TPE^{(2q,q)}(\gamma)
    \geq
    \pi^q
	\text{ for all }
	\gamma\in W^{1, 1}_{\mathrm{i,r}}\cap \{ \ArcLength(\gamma)=1\}.
\]
For $q>1$, the unique minimizer is the standard circle (up to similarities) and for $q=1$ all planar, convex embeddings attain the minimium.
\\
We now present an alternative proof for \aref{thm: circle min p=2q}, relying instead on the length bound on $u\mapsto \GaussMap_\gamma(u,w)$ (\aref{lem:LengthOfGaussMapinU}).
This also turns out to be the right approach for most $p>2q$. 
To this end, let $\gamma\in W^{1,\infty}_{\mathrm{i,a}}(\Domain, \R^n)$ and $w\in (0,1)$. We define
\begin{equation*}
    G^{(p,q)}_w(\gamma)
    \ceq
    \int_{\Domain}
        \left(
            \frac{1}{2}
            \abs{\frac{\partial}{\partial u} \GaussMap _\gamma (u,w)}
        \right)^q
            \abs{\gamma(u+w)-\gamma(u)}^{2q-p}
    \dd u
	\text{ and }
	G^{p,q}(\gamma)= \int_0^1 G_w^{(p,q)}(\gamma) \dd w
\end{equation*}
We observe, that for $q=1$, $G^{(p,q)}(\gamma)=\TP^{(p,q)}(\gamma)$ for any convex $\gamma$.

\begin{lemma}
\label{lem: G minorant for TP}
	Let $q \geq1$, $p>0$ and $\gamma \in W^{1,\infty}_{i,a}$. Then
	$G^{(p,q)}(\gamma) \leq \TP^{(p,q)}(\gamma) $, with equality if $\Image (\gamma)$ is a circle.
\end{lemma}

\begin{proof}
	Let $\gamma \in W^{1,\infty}_{\mathrm{i,a}}$.
	If $\TP^{(p,q)}(\gamma )= \infty$, there is nothing to show.
	Hence, we may  assume that $\TP ^{(p,q)}(\gamma )< \infty$. 
	Using \aref{eq:DerivGaussMapU}, convexity and a triangle inequality, we compute the following.
		\begin{align*}
			&\TP ^{(p,q)}(\gamma) 
			=
			\int _\Domain \int _\Domain 
				\fdfrac{\abs{P^\perp_{\gamma'(u)}(\gamma(u+w)-\gamma(u))}^q}
				{\abs{\gamma(u+w)-\gamma(u)}^p}
			\dd w \dd u
			\\
			=&
				\frac{1}{2} 
				\int _\Domain \int _\Domain 
					\fdfrac{\abs{P^\perp_{\gamma'(u+w)}(\gamma(u+w)-\gamma(u))}^q}
					{\abs{\gamma(u+w)-\gamma(u)}^p}
				\dd w \dd u
			+
				\frac{1}{2} 
				\int _\Domain \int _\Domain 
					\fdfrac{\abs{P^\perp_{\gamma'(u)}(\gamma(u+w)-\gamma(u))}^q}
					{\abs{\gamma(u+w)-\gamma(u)}^p}
				\dd w \dd u
			\\
			=&
			\frac{1}{2}
			\int _\Domain \int _\Domain 
			\left[ 
				\fdfrac{\abs{P^\perp_{\gamma'(u+w)}(\gamma(u+w)-\gamma(u))}^q}
				{\abs{\gamma(u+w)-\gamma(u)}^{2q}}
			+ 
				\fdfrac{\abs{P^\perp_{\gamma'(u)}(\gamma(u+w)-\gamma(u))}^q}
				{\abs{\gamma(u+w)-\gamma(u)}^{2q}}
			\right] 
			\abs{\gamma (u+w) - \gamma (u)}^{2q-p}
			\dd w \dd u
			\\
			\geq&
				\int _\Domain \int _\Domain 
				\left[ 
					\frac{1}{2}  	
					\fdfrac{\abs{P^\perp_{\GaussMap_\gamma (u+w,u)}\gamma'(u+w)}}{\abs{\gamma(u+w)-\gamma(u)}}
				\, + \, 
					\frac{1}{2}  	
					\fdfrac{\abs{P^\perp_{\GaussMap _\gamma (u+w,u)}\gamma '(u)}}{\abs{\gamma(u+w)-\gamma(u)}}
				\right] ^q
			 	\abs{\gamma (u+w) - \gamma (u) }^{2q-p}
				\dd w \dd u
			\\
			\geq&
				\int _\Domain \int _\Domain 
				\left[ 
					\frac{1}{2}  	
					\fdfrac{\abs{P^\perp_{\GaussMap _\gamma (u+w,u)}(\gamma'(u+w)- \gamma (u))}}
					{\abs{\gamma(u+w)-\gamma(u)}}
				\right] ^q
				\abs{\gamma (u+w) - \gamma (u) }^{2q-p}
				\dd w \dd u
			\\
			=&
				\int _\Domain \int _\Domain 
				\left(
					\frac{1}{2}
					\abs{\frac{\partial}{\partial u} \GaussMap (u,w)}
				\right)^q
				\abs{\gamma(u+w)-\gamma(u)}^{2q-p}
				\dd w \dd u
			=G^{(p,q)}(\gamma)
		\end{align*}
	If $\Image (\gamma)$ is a cirlce, both inequalities are equalities, 
	because $P^\perp_{\GaussMap _{\gamma} (u+w,u)}\gamma '(u+w)=-P^\perp_{\GaussMap _{\gamma} (u+w,u)}\gamma'(u)$ for all $u,w \in \Domain$.
\end{proof}
Being a special case of $\Funct_1$, the minorant $G^{(2q,q)}$ is minimized by the circle.
\begin{lemma}
\label{lem:gMinP=2q} 
    Let $\gamma \in W^{1,\infty}_{\mathrm{i,a}}$ and $q\geq 1$. 
    Then $G_w^{(2q,q)}(\gamma) \geq \pi^q$ for every $w \in (0,1)$ and $G^{(2q,q)}(\gamma) \geq \pi^q$.
    Furthermore, $G_w^{(2,1)}(\gamma)=\pi$ for all $w$ iff $\gamma $ is convex and for $q>1$, $G_w^{(2q,q)}(\gamma)=\pi^q$ for all $w$ iff $\Image(\gamma)$ is a circle.
\end{lemma}

\begin{proof}
	  We have
	  \[
	  G^{(2q,q)}_w(\gamma)=
	  \int_{\Domain}
	  \left(
	  \frac{1}{2}
	  \abs{\frac{\partial}{\partial u} \GaussMap (u,w)}
	  \right)^q
	  \dd u
	  \]
      For $q=1$, we may simply refer to \aref{lem:LengthOfGaussMapinU} in order to conclude that $G^{(2,1)}_w$ is minimal iff $\gamma$ is convex. 
	  For $q>1$, we refer to \aref{lem: convex f wu bound} with the choice $f(x)=x^q$ to see that $G^{(2q,q)}_w$ is uniquely minimized by the circle. 
	  Integrating with respect to $w$ yields the result for $G^{(2q,q)}$.
\end{proof}
\aref{thm: circle min p=2q} now follows from combining \aref{eq: minorant argument} with \aref{lem: G minorant for TP} and \aref{lem:gMinP=2q}.

%% file: GlobaleExtremavonTP.tex
\subsection{The case $p>2q$}
\label{section: The case p>2q}
The integrand of $G^{(2q,p)}_w$ consists of the two factors $\abs{\frac{\partial}{\partial u} \GaussMap _\gamma (u,w)}^q$ and $\abs{\gamma (u+w) - \gamma (u)}^{2q-p}$. 
If $p>2q$, then the power of the second factor is negative. 
Hence, upon integrating any of them separately with respect to $u$, the integrals are minimized by the circle (see \aref{lem:gMinP=2q} and~\cite{ABRAMS2003381}). 
However, in order to establish $G^{(p,q)}$ as a special case of $\Funct _2$, a suitable negative power of the second factor has to be a concave function. 
This yields the bound $p \leq4q-2$.
\begin{lemma}
\label{lem: TP_w min p >= 2q}
    Let $\gamma \in W^{1,\infty}_{\mathrm{i,a}}$, $q>1$, $p \in \intervaloc{2q, 4q-2}$ and $w \in (0,1)$ be arbitrary. 
    Then 
    \[ 
        G^{(p,q)}_w(\gamma) \geq \pi^{p-q} \sin(\pi w)^{2q-p} 
    \]
    and if $p<2q+1$, then
    \[
        G^{(p,q)}(\gamma) \geq \pi^{p-q} \int_0^1 \sin(\pi w)^{2q-p} \dd w 
        ,
    \]
    with equality iff $\Image(\gamma)$ is a circle
\end{lemma}

\begin{proof}
	We observe, that the second statement follows from the first one.
    Let $\gamma, q, p, w$ be as above.
    With the notation from \aref{lem:MinI4} set $f(x)=(\frac{1}{2}x)^q$, $g(x)=x^{(p-2q)/2}$ and $\theta = 1/q$. 
    Since $q>0$ and $p-2q>0$, $f$ and $g$ are strictly monotonically increasing. 
    Moreover, $f^\theta$ is convex and $g^{1/(\theta-1)}=(\cdot)^{(2q-p)/(2-2q)}$ is concave, as $\frac{p-2q}{2q-2} \in \intervaloc{0,1}$. 
    Now \aref{lem:MinI4} yields that
    \[
        G^{(p,q)}_w(\gamma) 
        \geq 
            f(2 \pi)   
            \frac{1}{g(\pi^{-2} 
            \sin (\pi w)^2)}
        = 
            \pi^q \, 
            \left( \frac{\sin(\pi w)}{\pi}\right)^{2q-p}
        ,
    \]
    with equality iff $\Image( \gamma )$ is a circle.
\end{proof}
Combining \aref{lem: TP_w min p >= 2q} with \aref{lem: G minorant for TP}, again by the standard argument \aref{eq: minorant argument} we conclude that circles minimize most $\TP^{(p,q)}$ in this regime.
We want to point out, that the Hilbert case ($q=2$, $p \in (4,5)$) is fully included in the parameter range of \aref{thm: circle min p geq 2q}.
\\
As mentioned before, we are not able to prove that all energies with $p\in \intervalco{2q, 2q+1}, q\geq 1$ are minimized by round circles.
Nonetheless, we are able to prove that, for the remaining parameters of this regime, the energies are minimized by circles among all convex curves.
In order to use similar techniques as before, we have to prove a somewhat ``degenerate'' case, where the energies are maximized by circles.

\begin{lemma}\label{lem:TPMaxByCircle}
    Let $\gamma \in W^{1,\infty}_{\mathrm{i,a}}$ be convex, $0\leq q\leq 1$ and 
    $w\in \intervaloo{0,1}$ be arbitrary. 
    Then $G^{(q,q)}_w(\gamma) \leq\sin(\pi w)^q$ with equality iff $\Image(\gamma)$ is a circle.
\end{lemma}
\begin{proof}
    Let $q,w, \gamma$ be given.
    Combining \aref{lem: max | Delta f| on S^1} and Jensen's inequality for concave functions, we compute the following.
    \begin{align*}
   G^{(q,q)}_w(\gamma)
    &=
        \int_{\Domain}
            (
                \frac{1}{2}
                \abs{P^\perp_{\GaussMap_\gamma}(\gamma'(u+w)-\gamma'(u))}
            )^q
        \dd u
    \leq
    \frac{1}{2^q}\int_{\Domain}
        \abs{\gamma'(u+w)-\gamma'(u)}^q
    \dd u
    \\
    &\leq
        \sin(\pi w)^q
    .
    \end{align*}
    In the above, equality holds, iff $\GaussMap_\gamma(u,w) \perp \gamma'(u+w)-\gamma'(u)$ for all $u \in \Domain$ and $\abs{\gamma ''}=2\pi$.
    Hence, equality holds iff $\Image(\gamma)$ is a circle.
\end{proof}

\begin{lemma}\label{lem:G MinConvexCurvesP>2Q}
    Let $\gamma \in W^{1,\infty}_{\mathrm{i,a}}$ be convex, $q\geq 1, p>2q $ and $w \in (0,1)$ be arbitrary.
    Then 
    \[
        G^{(p,q)}_w(\gamma) 
        \geq 
            \pi^{p-q}\sin(\pi w)^{2q-p}
    \]
    and if $p<2q+1$, then
    \[
        G^{(p,q)}(\gamma)
        \geq 
        \pi^{p-q}
        \int_0^1 
            \sin(\pi w)^{2q-p} 
        \dd w
    ,
    \]
    with equality iff $\Image(\gamma)$ is a circle.
\end{lemma}
\begin{proof}
	 Let $p,q,w, \gamma$ be as above. 
	We observe, that the second statement follows from the first one.
	Recall that we have
	\[
		G^{(p,q)}_w(\gamma) 
		=
		\int_\Domain
		\pars*{
			\frac{
				\abs{P^\perp_{\GaussMap_\gamma}(\gamma'(u+w)-\gamma'(u))}
			}{2\abs{\gamma(u+w)-\gamma(u)}}
		}^q
		\abs{\gamma (u+w) - \gamma (u)}^{2q-p}
		\dd u
		=: \int_\Domain h(u) \dd u
		.
	\]
	From the given $w$, we deduce, that there is a $c>0$ such that $c\geq \abs{\gamma(u+w)-\gamma(u)}\geq c^{-1}$ for all $u$.
	Hence, there holds that $G_w^{(p,q)}(\gamma)< \infty$.
	\\
	We define
	\[
	f(u) \ceq \frac{\abs{P^\perp_{\GaussMap_\gamma}(\gamma'(u+w)-\gamma'(u))}}{2\abs{\gamma(u+w)-\gamma(u)}} \quad \text{and} \quad 
	g(u) \ceq \abs{P^\perp_{\GaussMap_\gamma}(\gamma'(u+w)-\gamma'(u))} 
	\]
	Let further $\alpha=p-q>0$ and $\beta=2q-p<0$. Notice that if $g \neq0$, then $h$ is equal to $f^\alpha g^\beta$. 
    Thus, if we define for $n \in \N$ the function $g_n \ceq \max \{g, 1/n\} \leq 2$, then the sequence $(f^\alpha g_n^\beta)_{n \in \N}$ is monotonically increasing and converges pointwise to $h$. 
    Therefore, by monotone convergence
	\[
	G_w^{(p,q)}(\gamma)= \lim_{n \to \infty} 
	\int_\Domain f^\alpha g_n^\beta \dd u .
	\]
	Now, we may apply \aref{lem:RevHoelder} to $\int f^\alpha g_n^\beta \dd u$, in order to bound the right hand side from below, yielding  
	\begin{align*}
	G_w^{(p,q)}(\gamma) \geq&
	\lim_{n \to \infty}  \left( \left(\int_\Domain f \dd u \right)^{\alpha} \cdot \left( \int_\Domain g_n^{\frac{\beta}{1-\alpha}} \dd u \right)^{1- \alpha} \right)
	\end{align*}
	Notice, that $\int_\Domain g^{\frac{\beta}{1-\alpha}} \dd u \neq0$, because otherwise $g \equiv 0$ and thus also $\abs{\partial _u \GaussMap _\gamma }=  g / \abs{\gamma (\cdot +w) - \gamma (\cdot)}\equiv 0$, which can not be true for an embedded closed curve. 
    Therefore, by continuity of $x \mapsto x^{1-\alpha}$ on $\intervaloo{0, \infty}$ and dominated convergence we get
    \begin{align*}
		G_w^{(p,q)}(\gamma) \geq&
	\left(\int_\Domain f \dd u \right)^{\alpha} \cdot \left( \lim_{n \to \infty} \int_\Domain g_n^{\frac{\beta}{1-\alpha}} \dd u \right)^{1- \alpha} 
	\\
	=&
	\left(\int_\Domain f \dd u \right)^{\alpha} \cdot \left( \int_\Domain g^{\frac{\beta}{1-\alpha}} \dd u \right)^{1- \alpha} 
	\\
	=&
	\left(\int_\Domain \frac{1}{2} \abs{\partial _u \GaussMap _\gamma } \dd u \right)^{\alpha} \cdot \left( \int_\Domain \abs{P^\perp_{\GaussMap_\gamma}(\gamma'(u+w)-\gamma'(u))} ^{\frac{\beta}{1-\alpha}} \dd u \right)^{1- \alpha} .
	\end{align*}
	Since $\frac{\beta}{1 - \alpha} \in \intervaloc{0,1}$ we conclude by \aref{lem:TPMaxByCircle} and \aref{lem:LengthOfGaussMapinU} that
	\(
	G^{(p,q)}_w(\gamma) 
	\geq
	\pi^{p-q}
	\sin(\pi w)^{2q-p},
	\)
	with equality iff $\Image(\gamma)$ is a circle.
\end{proof}
Using \aref{lem: TP_w min p >= 2q} with \aref{lem: G minorant for TP} and \aref{lem:G MinConvexCurvesP>2Q}, we can finally state the main result of this section.
\begin{theorem}
\label{thm: circle min p geq 2q}
    Let $\gamma \in W^{1,\infty}_{\mathrm{i,a}}$ and $q>1, p \in \intervaloc{2q, 4q-2} \cap (2q,2q+1)$. Then 
    \[ 
    \TP^{(p,q)}(\gamma ) \geq \pi^{p-q} \int_{0}^{1} \sin(\pi w)^{2q-p} \dd w \, ,
    \]
    with equality iff $\Image( \gamma )$ is a circle.
    Furthermore, the inequality holds true for $q\geq1, p\in \intervaloo{2q,2q+1}$ among all convex curves out of $W^{1,\infty}_{\mathrm{i,a}}$ with equality iff $\Image(\gamma)$ is a circle.
\end{theorem}

%% file: LowerLimit.tex
\subsection{The lower limit case $p=q+1$}\label{section: The lower limit case p=q+1}
\noindent
Let's consider a rather unexpected parameter combination, namely $q \geq 1, p=q+1$. 
We refer to it as the \textit{lower limit case}, because for $p<q+1$ there is no minimizer among convex, embedded curves.
This can be seen, when a family of convex curves converges to a double transversed line segment, as its energy tends to zero.
We show that $\TPE^{(q+1,q)}$ is uniquely minimized by the circle for $q>1$.
In the last section, we will use this result to minimize the energy for $p \geq \max \{2q-2,q+1\}$.
\\
The lower limit case may not be particularly interesting from the viewpoint of geometric knot theory, because these energies are not self repulsive in the ordinary sense. 
However, the tangent-point energies still seem to exhibit some kind of repulsion in a weak sense for $p\in\intervaloo{q+1,q+2}, q\geq 1$. 
Even though the energy does not diverge for orthogonal intersections, as one can compute for a cross-section, the energy still penalizes tangential intersections of sufficiently high order.
We refer to this behavior as \textit{mild repulsion}. 
Unfortunately, a detailed investigation of these mildly repulsive energies would exceed the scope of this work, but we give a definition for its characterizing properties.

\begin{definition}
	Let $\gamma:\Domain\rightarrow \R^n$ be a continuous curve and $s >1$.
	We say, that $\gamma$ has a \textit{touching point of order s} at $u\in \Domain$, if there exists $v \neq u \in \Domain$, such that 
	\begin{itemize}
		\item $\gamma (u)=\gamma(v)$ 
		\item $\gamma '(u), \gamma '(v)$ exist and $\gamma '(u)=\gamma '(v) \neq 0$
		\item $\limsup_{x\rightarrow u}\frac{\abs{P^\top_{\gamma'(u)}(\gamma (x)- \gamma(u))}}{\abs{x-u}^s} <\infty \quad \text{ and } \quad  \limsup_{x\rightarrow v}\frac{\abs{P^\top_{\gamma'(u)}(\gamma(x)- \gamma(v))}}{\abs{x-v}^s} <\infty$
	\end{itemize}
    Let $X$ be a subspace of the space of continuous functions, equipped with a topology $\tau$ and $\mathcal{K}_X\ceq X \cap \{ \gamma \text{ an embedding }\}$.
	We say that a functional $F:\mathcal{K}_X\rightarrow \R$ is \textit{mildly s-repuslive} with respect to $\tau$, if $F(\gamma_n)\nearrow \infty$ for all sequences $(\gamma_n)\subset \mathcal{K}_X$, converging to an immersion with a touching point of order s with respect to $\tau$.	
\end{definition}
\begin{remark}
	If $\gamma \in C^{1, \alpha}$ for $ \alpha \in \intervaloc{0,1}$, then every tangential intersection point, i.e. $u,v$ with $\gamma (u)=\gamma (v)$ and $\gamma '(u)=\gamma '(v)$, is at least of order $1+ \alpha$.
\end{remark}
Furthermore, for planar curves, they are closely related to other non-local, geometric functionals arising from the study of the perimeter or the so-called \textit{fractional perimeter} of planar domains. For example, if $\gamma$ parametrizes the boundary of a set $\Omega \subset \mathbb{R}^n$, $n=2$, then the energy $F^{(p,q)}(\gamma)$, introduced below, for $p=3$, $q=2$ coincides with the functional 
	\[
	\int_{\partial \Omega \times \partial \Omega} 
	    \frac{\abs{\inner{ n(x),y-x\rangle \langle x-y, n(y) } }}{\norm{x-y}^{n+1}} 
    \dd \sigma (x) \dd \sigma (y) 
	\]
considered in~\cite{steinerberger2022inequality} and later in~\cite{bushling2024singular}.
Similar terms occurred in~\cite{o2015mobius} and~\cite{o2018regularized}.
Our methods provide an alternative proof for~\cite{steinerberger2022inequality} in the case $n=2$, slightly improving the initially required regularity.
On the other hand, the techniques used in~\cite{bushling2024singular} are more general than ours.
Nonetheless, our techniques can be used as geometric intuition.
\\
The proof that circles minimize the energy in the lower limit case is broadly similar to the one from the previous section. 
Again we present a suitable minorant, which this time is related to the functional $\Funct _3$ from \aref{thm:MinI5}. 
To this end, let $\gamma \in W^{1, \infty}_{\mathrm{i,a}}(\Domain, \mathbb{R}^n)$ and $u \in \Domain$. 
We define $\alpha=p-q, \beta=2q-p$ and
\begin{align*}
        F_u^{(p,q)} (\gamma )
        \ceq&
            \int_{0}^{1}
                    \left( 
                        \frac{
                            \abs{
                                P^\perp_{\gamma ^\prime (u)} (\gamma(u+w)- \gamma (u))
                            }}
                        {\abs{ \gamma(u+w)-\gamma(u)}} \right)^\beta  
                        \left( 
                        \frac{
                        	\abs{
                        		P^\perp_{\gamma ^\prime (u+w)} (\gamma(u+w)-\gamma(u))
                        	}
                        }
                        {\abs{ \gamma(u+w)-\gamma(u)}^2} 
                        \right)^\alpha
                        \dd w 
        \\
        =&
           \int_{0}^{1}
           \abs{P^\perp _{\gamma ^\prime (u)} \GaussMap_\gamma (u,w)}^\beta
           \abs{\frac{\partial}{\partial w} \GaussMap_\gamma (u,w)}^\alpha
           \dd w
        ,
        \\
    F^{(p,q)}(\gamma)
        \ceq&\int_{\Domain} F_u^{(p,q)} (\gamma) \dd u
    ,
    \end{align*}
where we set $F_u^{(p,q)}(\gamma)=\infty$, if $\gamma ^\prime (u)$ does not exist.
\begin{lemma}\label{lem: F^pq TP^pq relation}
    Let $q>0, p\geq0$ and $\gamma \in W^{1,\infty}_{\mathrm{i,a}}$.
    If $p=q$ or $p=2q$, then $\TP^{(p,q)} \equiv F^{(p,q)}$. 
    If $p \in \intervaloo{q,2q}$, then $F^{(p,q)}(\gamma) \leq \TP^{(p,q)}(\gamma)$.
    Furthermore, among convex curves, equality holds if and only if $\Image(\gamma)$ is a circle.
\end{lemma}

\begin{proof} 
If $p=q$ then $\alpha =0$ and if $p=2q$ then $\beta =0$, so the claim follows immediately.
Hence, let $q>0, p \in \intervaloo{q,2q}$ and $\gamma\in W^{1,\infty}_{\mathrm{i,a}}$ be arbitrary.
Since $\alpha=p-q, \beta=2q-p$, we conclude that $\alpha , \beta >0$. 
Applying Young's inequality, with $q_1 = \frac{\alpha + \beta}{\alpha}$ and $p_1=q_1^\prime$ its Hölder conjugate, yields
\begin{equation*}
    a^\alpha \, b ^\beta \leq \frac{1}{q_1} a^{\alpha + \beta} + \frac{1}{p_1} b^{\alpha + \beta} \,
    \text{for all }
    a,b\geq 0 .
\end{equation*}
Let $\gamma \in W^{1,1}_{\mathrm{i,r}}$ be arbitrary.
Applying the above inequality to 
$a=\abs{P^\perp_{\gamma ^\prime (u+w)}(\gamma(u+w)-\gamma(u))}$ and 
$b=\abs{P^\perp_{\gamma ^\prime (u)} (\gamma(u+w)-\gamma(u))}$, yields
\begin{align*}
&F^{(p,q)}(\gamma )
=
    \int_\Domain
        \int_{0}^{1}
           \abs{P^\perp _{\gamma ^\prime (u)} \GaussMap_\gamma (u,w)}^\beta
           \abs{\frac{\partial}{\partial w} \GaussMap_\gamma (u,w)}^\alpha
        \dd w
    \dd u
\\
    &\leq
        \int_{\Domain} 
            \int_{0}^{1} 
                \frac{1}{q_1}
                    \left( 
                        \fdfrac{
                            \abs{
                                P^\perp_{\gamma ^\prime (u+w)} (\gamma(u+w)-\gamma(u))
                            }^{\alpha+ \beta}
                        }
                        {
                            \abs{ 
                                \gamma(u+w)-\gamma(u)
                            }^{2 \alpha + \beta}
                        } \right) 
                + 
                 \fdfrac{1}{p_1}
                    \left( 
                        \frac{
                            \abs{
                                P^\perp_{\gamma ^\prime (u)} (\gamma(u+w)-\gamma(u)) 
                            }^{\alpha+ \beta}}
                        {\abs{ \gamma(u+w)-\gamma(u)}^{2 \alpha + \beta}} 
                    \right) 
        \dd w \dd u
    \\
    &=
        \TPE^{(2\alpha+\beta, \alpha+\beta)}(\gamma)
    =
        \TPE^{(p,q)}(\gamma)
\end{align*}
In the above, equality holds, if and only if 
\begin{equation}
    \label{eq: Circle property}
        \abs{
            P^\perp_{\gamma ^\prime (u+w)} (\gamma(u+w)-\gamma(u))
        }
        =
        \abs{
            P^\perp_{\gamma ^\prime (u)} (\gamma(u+w)-\gamma(u))
         }
        \text{ for a.e. } u \in \Domain, \; w \in (0,1)
    .
\end{equation}
Using \aref{thm:UmlaufwinkelsatzProjektorEig}, we conclude that, among all convex curves, \aref{eq: Circle property} is true iff $\Image(\gamma)$ is a circle.
\end{proof}

\begin{remark}
    Since for $p\in \intervalco{0,2q+1}\setminus [q,2q]$, the reverse inequality of \aref{lem: F^pq TP^pq relation} holds true,
    the family $F^{(p,q)}$ is not suited to determine whether these tangent-point energies are minimized by the circle.
    Hence, if the circle maximizes $F^{(p,q)}$ in this regime of parameters, then it is the unique global maximizer of $\TPE^{(p,q)}$.
\end{remark}
We now show similar bounds for $F_u^{(q+1,q)}$, as we did for $G_w^{(p,q)}$ in \aref{lem: TP_w min p >= 2q}. 
Here, the unique property of the lower limit case is that $\alpha=1$.
This enables us to view $F^{(q+1,q)}$ as a special case of the functional $\Funct_3$ from \aref{thm:MinI5}.

\begin{lemma}\label{lem: Minimizer of S p=q+1}
    Let $\gamma \in W^{1,\infty}_{\mathrm{i,a}}$ and $q >0$. 
    Then 
    \[
        F_u^{(q+1,q)} (\gamma) \geq \pi \int_0^1 \sin (\pi w)^{q-1} \dd w
        \text{ for almost every } u \in \Domain,
    \]
    with equality for almost every $u$ iff $\gamma$ is convex.
    Furthermore,
    $F^{(q+1,q)}(\gamma)\geq \pi \int_0^1 \sin (\pi w)^{q-1} \dd w$, where equality holds iff $\gamma$ is convex.
\end{lemma}
\begin{proof}
	We observe, that the second statement follows from the first one.
    Let $\gamma \in W^{1,\infty}_{\mathrm{i,a}}$. 
    Recall that
    \[
        F_u^{(q+1,q)}(\gamma)
        =
        \int_{0}^{1}
            \abs{P^\perp _{\gamma ^\prime (u)} \GaussMap_\gamma (u,w)}^\beta
            \abs{\frac{\partial}{\partial w} \GaussMap_\gamma (u,w)}
        \dd w \, ,
    \]
    where $q>0$ implies that $\beta >-1$.
    In order to apply \aref{lem:MinI5}, we set $g(x)=(1-x^2)^{\beta/2}$ and $\varepsilon= \frac{\beta +1}{2}$. 
    Since $\frac{\beta}{2} \cdot (2+\varepsilon)>\beta - \frac{\varepsilon}{2}=\frac{3}{4}\beta -\frac{1}{4}>-1$, we have that $g \in L^{2+ \varepsilon}((-1,1),\intervalco{0, \infty})$. 
    Moreover,
    \[
        g(\inner{ t_\gamma (u), \GaussMap _\gamma (u,w)} )
    =
        (
            1-\inner{ \gamma ^\prime (u), \GaussMap _\gamma (u,w)}^2
        )^{\frac{\beta}{2}}
    =
    \abs{P^\perp _{\gamma ^\prime (u)} \GaussMap _\gamma (u,w)}^\beta
    .
    \] 
    Therefore, \aref{lem:MinI5} implies that
    \[
        F_u^{(q+1,q)}(\gamma) 
    \geq 
        \pi \int_0^1 g(\cos (\pi w)) \dd w
    =
        \pi \int_0^1 (1-\cos (\pi w)^2)^{\frac{\beta}{2}} \dd w
    =
        \pi \int_0^1 \sin (\pi w)^{q-1} \dd w
    ,
    \]
    with equality for almost every $u \in \Domain$ iff $\gamma$ is convex.
\end{proof}
Similarly as before, we combine \aref{lem: F^pq TP^pq relation} with \aref{lem: Minimizer of S p=q+1}, in order to conclude the following.

\begin{theorem}\label{thm: circle min p=q+1}
    Let $\gamma \in W^{1,\infty}_{\mathrm{i,a}}$ and $q >1$. 
    Then 
    \[
    \TP^{(q+1,q)}(\gamma) \geq \pi \int_0^1 \sin (\pi w)^{q-1} \dd w,
    \]
    with equality iff $\Image( \gamma)$ is a circle.
\end{theorem} 
Note that, in order to apply the second part of \aref{lem: F^pq TP^pq relation}, we need $q+1 \in (q,2q)$ and thus $q > 1$.
Furthermore, \aref{lem: Minimizer of S p=q+1} rules out any non-convex curve as a minimizer and \aref{lem: F^pq TP^pq relation} rules out any convex curve, of which the image is not a circle. 
The combination of both yields the uniqueness of the circle as a minimizer of $\TP^{(q+1,q)}$, as claimed above.

\subsection{The case $p<2q$}\label{section: p<2q}
In this section, we investigate the global minimizer of the tangent-point energies for $p<2q$. 
Unfortunately, for $p \in (q+1,2q)$, neither $G_w$ nor $F_u$ are minimized by the circle.
However, similarly to $p >2q$, where we fell back to the geometric case $p=2q$, we minimize $\TPE^{(p,q)}$ for some $p \in (q+1,2q)$ by falling back to the case $p=q+1$. 
Our approach, has the somewhat counterintuitive drawback that it does only work for $p$ being slightly smaller than $2q$.

	\begin{theorem}\label{thm: circle minimizer for p>=2q-2}
		Let $\gamma \in W^{1,\infty}_{\mathrm{i,a}}$ and $q>1, p \in \intervalco{2q-2,2q} \cap (q+1,2q)$. 
		Then
		\[
		\TP^{(p,q)}(\gamma)
		\geq 
		\pi ^{p-q} 
		\int _0^1 
		\sin (\pi w)^{2q-p} 
		\dd w \, ,
		\]
		with equality iff $\Image ( \gamma )$ is a circle.
	\end{theorem}
	
	\begin{proof}
		Define 
		\[
		\sigma
		\ceq 
		(2q-p)  
		\frac {q}{2q-p+1} 
		\; 
		,   
		\quad
		\mu 
		\ceq (2q-p)  
		\frac{q-p+1}{2q-p+1}
		\]
		A short computation yields $\sigma + \mu = 2q-p \in \intervaloc{0,2}$, $\frac{\sigma }{\sigma + \mu } + \sigma= q$ 
		and $\fdfrac{2\sigma}{\sigma + \mu} + \sigma - \mu =p$.
		Furthermore, we know that $\sigma >0$ and $\mu<0$. 
		Let 
		\[
		q_1
		\ceq 1 + \frac{1}{\sigma + \mu} 
		\text{, } 
		p_1
		\ceq 1+ \sigma + \mu 
		\text{, }
		q_2
		\ceq \frac{\sigma + \mu}{\sigma } 
		\text{ and }
		p_2
		\ceq \frac{\sigma+ \mu}{\mu} 
		\]
		One computes $p_1,q_1>1$, $q_2\in (0,1)$, $p_2<0$ and $\frac{1}{p_i}+ \frac{1}{q_i}=1$.
		Note that $\sigma  q_1= \frac{\sigma}{\sigma + \mu} p_1=q$. 
		Let $\gamma \in W^{1,1}_{\mathrm{i,a}}$.
		For the sake of short notation, we abbreviate $\Delta_{u,u+w}\gamma\ceq \gamma(u+w)-\gamma(u)$.
		Using Young's inequality, one obtains
		\begin{align*}
			\TP^{(p,q)}(\gamma)
			&= 
			\frac{1}{q_1} 
			\int _{\Domain} 
			\int _{0}^{1}
			\frac{
				\abs{
					P^\perp _{\gamma ^\prime (u)} \Delta_{u+w,u} \gamma
				}^q
			}
			{
				\abs{\Delta_{u+w,u} \gamma }
				^p
			}
			\dd w 
			\dd u
			+  
			\frac{1}{p_1}
			\int _{\Domain} 
			\int _{0}^{1}
			\frac{
				\abs{
					P^\perp _{\gamma ^\prime (u+w)} \Delta_{u+w,u} \gamma
				}
				^q
			}
			{
				\abs{\Delta_{u+w,u} \gamma }^p
			}
			\dd w  
			\dd u
			\\
			&=
			\int _{\Domain} 
			\int _{0}^{1}
			\frac{
				\frac{1}{q_1} 
				\abs{
					P^\perp _{\gamma ^\prime (u)} \Delta_{u+w,u} \gamma
				}^{\sigma q_1}  
				+ 
				\frac{1}{p_1} 
				\abs{
					P^\perp _{\gamma ^\prime (u+w)} \Delta_{u+w,u} \gamma
				}^{\frac{\sigma}{\sigma + \mu} p_1}
			}
			{
				\abs{\Delta \gamma }^p
			}
			\dd w 
			\dd u
			\\
			&\geq
			\int _{\Domain}
			\int _{0}^{1} 
			\frac{
				\abs{
					P^\perp _{\gamma ^\prime (u+w)} \Delta_{u+w,u} \gamma
				}^\frac{\sigma}{\sigma + \mu}  
				\abs{
					P^\perp _{\gamma ^\prime (u)} \Delta_{u+w,u} \gamma
				}^\sigma
			}
			{
				\abs{
					\Delta_{u+w,u} \gamma  
				}^p
			} 
			\dd w  
			\dd u
			\\
			&=
			\int _{\Domain} 
			\int _{0}^{1}  
			\left(
			\frac{
				\abs{
					P^\perp _{\gamma ^\prime (u+w)} \Delta_{u+w,u} \gamma
				}
			}
			{
				\abs{\Delta_{u+w,u} \gamma}^2
			}
			\right)^\frac{\sigma}{\sigma + \mu}  
			\left(
			\frac{
				\abs{
					P^\perp _{\gamma ^\prime (u)} \Delta_{u+w,u} \gamma
				}
			}
			{
				\abs{\Delta_{u+w,u} \gamma  }
			}
			\right)^\sigma 
			\abs{
				\Delta_{u+w,u} \gamma 
			}^\mu 
			\dd w  
			\dd u
		\end{align*}
		Here, Young's inequality is an equality for the circle, because $\abs{P^\perp _{\gamma ^\prime (u)} \Delta_{u+w,u} \gamma}=\abs{P^\perp _{\gamma ^\prime (u+w)} \Delta_{u+w,u} \gamma}$.
		Using the inverse Young inequality (i.e. \aref{lem:RevYoung}), we further estimate $\TP^{(p,q)}(\gamma)$ as follows.
		\begin{align*}
			&\TP^{(p,q)}(\gamma) 
			\geq
			\int _{\Domain} 
			\int _{0}^{1}  
			\left(
			\frac{
				\abs{
					P^\perp _{\gamma ^\prime (u+w)} \Delta_{u+w,u} \gamma
				}
			}
			{
				\abs{\Delta_{u+w,u} \gamma }^2
			}
			\right)^\frac{\sigma}{\sigma + \mu}  
			\left(
			\frac{
				\abs{
					P^\perp _{\gamma ^\prime (u)} \Delta_{u+w,u} \gamma
				}
			}
			{
				\abs{
					\Delta_{u+w,u} \gamma  
				}
			}
			\right)^\sigma 
			\abs{
				\Delta_{u+w,u} \gamma 
			}^\mu 
			\dd w  
			\dd u
			\\
			&=
			\pi ^{
				-\mu  
				\frac{\sigma + \mu +1}{\sigma + \mu}
			}
			\int _{\Domain} 
			\int _{0}^{1}  
			\left(
			\frac{
				\abs{
					P^\perp _{\gamma ^\prime (u+w)} \Delta_{u+w,u} \gamma
				}
			}
			{
				\abs{
					\Delta_{u+w,u} \gamma 
				}^2
			}
			\right)^\frac{\sigma}{\sigma + \mu}  
			\left(
			\frac{
				\abs{
					P^\perp _{\gamma ^\prime (u)} \Delta_{u+w,u} \gamma
				}
			}
			{
				\abs{
					\Delta_{u+w,u} \gamma  
				}
			}
			\right)^\sigma 
			\left(
			\abs{\Delta_{u+w,u} \gamma }  
			\pi ^\frac{\sigma + \mu +1}{\sigma + \mu} 
			\right)^\mu 
			\dd w  
			\dd u 
			\\
			&\geq
			\pi ^{
				-\mu  
				\frac{\sigma + \mu +1}{\sigma + \mu}
			} 
			\int _{\Domain} 
			\int _{0}^{1} 
			\frac{1}{q_2}  
			\left[ 
			\left(
			\frac{
				\abs{
					P^\perp _{\gamma ^\prime (u+w)} \Delta_{u+w,u} \gamma
				}
			}
			{
				\abs{
					\Delta_{u+w,u} \gamma 
				}^2
			}
			\right)^\frac{\sigma}{\sigma + \mu}  
			\left(
			\frac{
				\abs{
					P^\perp _{\gamma ^\prime (u)} \Delta_{u+w,u} \gamma
				}
			}
			{
				\abs{
					\Delta_{u+w,u} \gamma  
				}
			}
			\right)^\sigma 
			\right]^{q_2} 
			\dd w
			\dd u
			\\
			&\quad+  
			\pi ^{
				-\mu  
				\frac{\sigma + \mu +1}{\sigma + \mu}
			} 
			\int _{\Domain} 
			\int _{0}^{1}
			\frac{1}{p_2} 
			\left( 
			\abs{
				\Delta_{u+w,u} \gamma 
			}  
			\pi ^\frac{\sigma + \mu +1}{\sigma + \mu} 
			\right)^{\mu \cdot p_2} 
			\dd w 
			\dd u
			\\
			&=
			\pi ^{
				- \mu 
				\frac{\sigma + \mu +1}{\sigma + \mu}
			} 
			\frac{1}{q_2}
			\int _{\Domain} 
			\int _{0}^{1} 
			\left(
			\frac{
				\abs{
					P^\perp _{\gamma ^\prime (u+w)} \Delta_{u+w,u} \gamma
				}
			}
			{
				\abs{
					\Delta_{u+w,u} \gamma 
				}^2
			}
			\right) 
			\left(
			\frac{
				\abs{
					P^\perp _{\gamma ^\prime (u)} \Delta_{u+w,u} \gamma
				}
			}
			{
				\abs{
					\Delta_{u+w,u} \gamma  
				}
			}
			\right)^{\sigma + \mu} 
			\dd w 
			\dd u
			\\
			&\quad
			+
			\pi ^{
				[
				(\sigma + \mu) - \mu
				]  
				\frac{\sigma + \mu +1}{\sigma + \mu}} 
			\frac{1}{p_2}
			\int _{\Domain} 
			\int _{0}^{1}
			\abs{
				\Delta_{u+w,u} \gamma 
			}^{\sigma + \mu} 
			\dd w 
			\dd u 
			\\
		\end{align*}
		By the choice of $\sigma$ and $\mu$, again equality is attained for the circle.\\
		Since the first integral is equal to $F^{(\sigma +\mu +2, \sigma + \mu +1)}(\gamma)$, we apply \aref{lem: Minimizer of S p=q+1} and conclude that it is uniquely minimized by the circle. For the second integral, since $\sigma + \mu \in \intervaloc{0,2}$, we use \aref{lem: circle max of Delta gamma |^r} and conclude that it is uniquely maximized by the circle.
		Finally, $p_1>0$ and $p_2<0$ imply that $\TPE^{(p,q)}$ is uniquely minimized by the circle.
\end{proof}	
\begin{remark}\label{rmk:LimitsOfOurMethods}
	Note that the range of $p$ in \aref{thm: circle min p geq 2q} and \aref{thm: circle minimizer for p>=2q-2} can be extended to $p\in \intervaloo{2q-c^*, 2q}\cap (q+1,2q)$ and $p\in \intervalco{2q, 2q+c^*(q-1)}\cap \intervaloo{2q,2q+1}$ respectively, 
	where $c^*$ is the maximal exponent $r$ such that \aref{thm: Wirtinger type inequality}, respectively \aref{lem: circle max of Delta gamma |^r}, hold true.
	By~\cite{ABRAMS2003381} and~\cite{exner2007critical}, we know that $c^*\in \intervalcc{2,2.5}$.
\end{remark}
\section{Minimizing Fractional Willmore energies}\label{sec:NonlocalWillmore}
We also want to highlight the close relation between $\TP^{(p,q)}$ and the fractional Willmore energies.
For $s\in (0,1), p\geq 1$ and a sufficiently nice, open set $E\subset \R^n$, they are given by 
\[
    \mathcal{W}_{s,p}(E)
    =
        \int_{\partial E}
            \abs{H_{\partial E, s}(x)}^p
        \dd \Hmeasure^{n-1}(x)
    ,
\]
where 
\[
    H_{\partial E,s}(x)
    = 
    \pm c_s
    \int_{\partial E}
        \frac{\inner{n(y), x-y}}{\abs{x-y}^{n+1+s}}
    \dd \Hmeasure^{n-1}(y)
\]
is the nonlocal mean curvature, introduced by~\cite{CaffarelliSavinRoquejoffreNonLocalMinSurfaces}.
The sign in front of the integral depends on whether the normal is pointing inward or outward.
In recent years, such nonlocal energies have become an active field of research.
Especially, fractional Willmore and bending energies have been investigated.
For example, in~\cite{blatt2025fractional}, the authors investigated regularizing effects of such energies and proved a number of structural theorems.
More recently, Giacomin and Schikorra investigated the scale-invariant case for convex subsets in $\R^2$ and were able to show that in this setting, minimizers exist (see~\cite{giacomin2024convex}).
Since in this case, the boundary of the set is a convex curve, we are able to show that these energies are minimized by round circles. 
In order to do so, we observe that the following holds true for convex curves $\gamma\in W^{1,\infty}_{\mathrm{i,a}}(\Domain, \R^2)$.
\begin{align*}
    H_{\partial E,s}(x)
    &=
    \pm c_s \int_\Domain
        \frac{
            \inner{n(\gamma(y)), \gamma(x)-\gamma(y)}
        }{\abs{\gamma(y)- \gamma(x)}^{2+s}}
    \dd y
    =
    c_s
    \int_\Domain
        \frac{
            \abs{P^\perp_{\gamma'(u+w)}(\gamma(u+w)-\gamma(u))}
        }{\abs{\gamma(u+w)-\gamma(u)}^{2+s}}
    \dd w
    \\
    &=
    c_s
    \int_\Domain
        \abs{\partial_w \GaussMap_\gamma}\abs{\gamma(u+w)-\gamma(u)}^{-s}
    \dd w
\end{align*}
It now becomes clear, that $(\TP^{(2+s,1)})^p$ is a suitable lower and $\TP^{(p(2+s),p)}$ a suitable upper bound for $\mathcal{W}_{s,p}$ among the class of convex curves.
\begin{proof}[of \aref{thm:NonlocalWillmoreMinByCircle}]
    Let $ s\in \intervaloo{0,1}$, $p\geq 1$ and $E\in X=\{E\subset \R^2 \text{ convex }\}$ be given. 
    Then there exists $\gamma\in W^{1,1}_{\mathrm{i,r}}(\Domain, \R^2)\cap \{\gamma \text{ convex}\}$ such that $\partial E=\Image(\gamma)$.
    Since $\mathcal{W}_{s,p}$ is geometric and $\mathcal{W}_{s,p}(\lambda \gamma)= \lambda ^{1-sp} \mathcal{W}_{s,p}(\gamma)$ for all $\lambda>0$, we assume $\gamma\in W^{1,\infty}_{\mathrm{i,a}}$.
    Using Jensen, we conclude that
    \begin{align*}
        \mathcal{W}_{s,p}(\gamma)
        &\geq
        \pars*{
            \int_{\Domain}
            \int_{0}^{1}
                \frac{\abs{P^\perp_{\gamma'(u+w)}(\gamma(u+w)-\gamma(u))}}{\abs{\gamma(u+w)-\gamma(u)}^{2+s}}
            \dd w \dd u
        }^p
        =
        (\TP^{2+s,1}(\gamma))^p
    \end{align*}
    Applying \aref{thm: circle min p geq 2q}, we obtain
    \[
         \mathcal{W}_{s,p}(\gamma)
        \geq
        \pars*{
            \pi^{1+s}
            \int_0^1
                \sin(\pi w)^{-s}
            \dd w
        }^p
    ,
    \]
    with equality iff $\Image(\gamma)$ is a circle.
\end{proof}

\section*{Acknowledgments}
The authors want to thank Heiko von der Mosel, Philipp Reiter and Simon Blatt for their invaluable feedback during the preparation of this article.
Elias Döhrer gratefully acknowledges the funding support from the European Union and the Free State of Saxony (ESF).
Alexander Dohmen is funded by the DFG-Graduiertenkolleg Energy, Entropy, and Dissipative Dynamics (EDDy), project no. 320021702/GRK2326.

%% file: appendix.tex
\section{Appendix}\label{sec:appendix}
First, we state the circumference angle theorem of circles, put into a more modern language.
The original version can be found in \cite[book 3, Proposition 26]{euclid1956elements}.
\begin{theorem}[\cite{euclid1956elements}]
    \label{thm:Umlaufwinkelsatz}
   Let $\gamma$ be a circle and let $x,y$ be two arbitrary, distinct points on $\Image(\gamma)$.
   Then $\gamma$ satisfies the following property (P) for $x,y$:
   The angle of the line segments $\overline{xz}$, respectively $\overline{yz}$, for a third point $z\in \Image(\gamma)$, is constant on the arcs of $\Image(\gamma)$ connecting $x$ and $y$.
\end{theorem}
Among convex curves the converse also holds true.
\begin{lemma}\label{lem:UmlaufWinkelSatzScharfBeiKreis}
    Let $\gamma\in W^{1,1}_{\mathrm{i,r}}$ be a convex curve.
    If $\gamma$ satisfies the property (P) for almost all $x,y$ and almost all third points, then $\Image(\gamma)$ is a circle
\end{lemma}
\begin{proof}
    Let $\gamma\in W^{1,1}_{\mathrm{i,r}}$ be convex and $x\neq y\in \Domain$ be arbitrary, such that (P) is satisfied for $\gamma(x),\gamma(y)$ and almost all third points.
    Since $\gamma$ is a convex embedding, we may assume that $\gamma\vert_{[x,y]}$ is not a straight line.
    By convexity, this chord is contained in one of the halfspaces, induced by the secant.
    Furthermore, we know that 
    \[
        c=
        \Angle{
            \frac{\gamma(z)-\gamma(x)}{\abs{\gamma(z)-\gamma(x)}}
        }{
            \frac{\gamma(z)-\gamma(y)}{\abs{\gamma(z)-\gamma(y)}}
        }
        \text{ for almost all } z\in (x,y)
        .
    \]
    Since $\gamma$ is absolutely continuous and injective, this map is continuous for $z\in(x,y)$.
    Hence, (P) holds true for all $\gamma(z)\in \Image(\gamma\vert_{(x,y)})$.
    Using a Moebius transformation, we directly conclude that $\gamma\vert_{[x,y]}$ is a circular arc.
    Repeating this construction with $z_1\in B_{\epsilon}(\frac{x+y}{2})$ and $z_2\in B_\epsilon(\frac{x+y+1}{2})$ now yields that $\gamma$ is a circle.
\end{proof}

\begin{theorem}\label{thm:UmlaufwinkelsatzProjektorEig}
    Let $\gamma \in W^{1,1}_{\mathrm{i,r}}(\Domain, \R^2)$ be convex.
    Then for almost all $u,w\in \Domain$
    \begin{equation}
        \label{eq:ProjektorEig}
        \abs{
            P^\perp_{\gamma'(u+w)}(
                \frac{
                    \gamma(u+w)-\gamma(u)
                }{
                    \abs{\gamma(u+w)-\gamma(u)}
                }
            )
        }
        =
        \abs{
            P^\perp_{\gamma'(u)}(
                \frac{
                    \gamma(u+w)-\gamma(u)
                }{
                    \abs{\gamma(u+w)-\gamma(u)}
                }
            )
        }
    .
    \end{equation}
    if and only if $\gamma$ satisfies the property (P) for almost all points.
    If one of the two statements holds true, then $\Image(\gamma)$ is a circle.
\end{theorem}
\begin{proof}
    Since $\gamma\in W^{1,1}_{\mathrm{i,r}}(\Domain, \R^2)$ is convex and \aref{eq:ProjektorEig} is geometric, we may reparametrize and rescale in order to assume that $\gamma\in W^{1,\infty}_{\mathrm{i,a}}$.
    If $\gamma$ satisfies (P) almost everywhere, 
    then we can apply \aref{lem:UmlaufWinkelSatzScharfBeiKreis} and conclude that $\gamma$ is a circle.
    Hence, one computes that \aref{eq:ProjektorEig} holds for all $u,w \in \Domain$.\\
    Now suppose that $\gamma$ satisfies \aref{eq:ProjektorEig} for almost every $u,w \in \Domain$.
    We now chose $u,w$ such that $\gamma'(u), \gamma'(u+w)$ exist and claim that $\gamma$ satisfies (P) for $u$, $u+w$. 
    Let $L$ be a secant through $\Image(\gamma)$, intersecting $\Image(\gamma)$ at 
    the points $\gamma(u)$ and $\gamma(u+w)$.
    We denote $\mathcal{I}_{u,u+w}= \Image(\gamma\vert_{(u,u+w)})$ and $\mathcal{I}_{u+w,u}=\Image(\gamma)\setminus \overline{\mathcal{I}_{u,u+w}}$ the arcs of $\Image(\gamma)$, connecting $\gamma(u)$ and $\gamma(u+w)$.
    It suffices to show, that for a third point $\gamma(z)\in \Image(\gamma)\setminus\{\gamma(u), \gamma(u+w)\}$, that the angle $\phi(u,u+w,z)$ of the segments $\overline{\gamma(z)\gamma(u+w)}$ and $\overline{\gamma(z)\gamma(u)}$ only depends on the choice of the arc, where $\gamma(z)$ is located.
    
    Using \aref{eq:ProjektorEig}, we know that 
    \[
            \alpha
            =
            \Angle{
                \frac{\gamma(u+w)-\gamma(u)}
                {\abs{\gamma(u+w)-\gamma(u)}}
            }
            {\gamma'(u+w)}
            =
            \Angle{
                \gamma'(u)
            }
            {
                \frac{\gamma(u+w)-\gamma(u)}
                {\abs{\gamma(u+w)-\gamma(u)}}
            }
    \]
    Now we choose a third point $z\in (u,u+w)$, such that the tangent $\gamma'(z)$ exists.
    Applying \aref{eq:ProjektorEig}, we observe that
    \[
        \theta_1
        =
        \Angle{
                \frac{\gamma(z)-\gamma(u)}
                {\abs{\gamma(z)-\gamma(u)}}
            }
            {\gamma'(z)}
            =
            \Angle{
                \gamma'(u)
            }
            {
                \frac{\gamma(z)-\gamma(u)}
                {\abs{\gamma(z)-\gamma(u)}}
            }
    \]
    and that 
    \[
        \theta_2
        =
        \Angle{
                \frac{\gamma(u+w)-\gamma(z)}
                {\abs{\gamma(u+w)-\gamma(z)}}
            }
            {\gamma'(u+w)}
            =
            \Angle{
                \gamma'(z)
            }
            {
                \frac{\gamma(u+w)-\gamma(z)}
                {\abs{\gamma(u+w)-\gamma(z)}}
            }
    .
    \]
    The following picture illustrates the relation of the angles.
\begin{center}
\includegraphics[height=6cm, width=10cm]{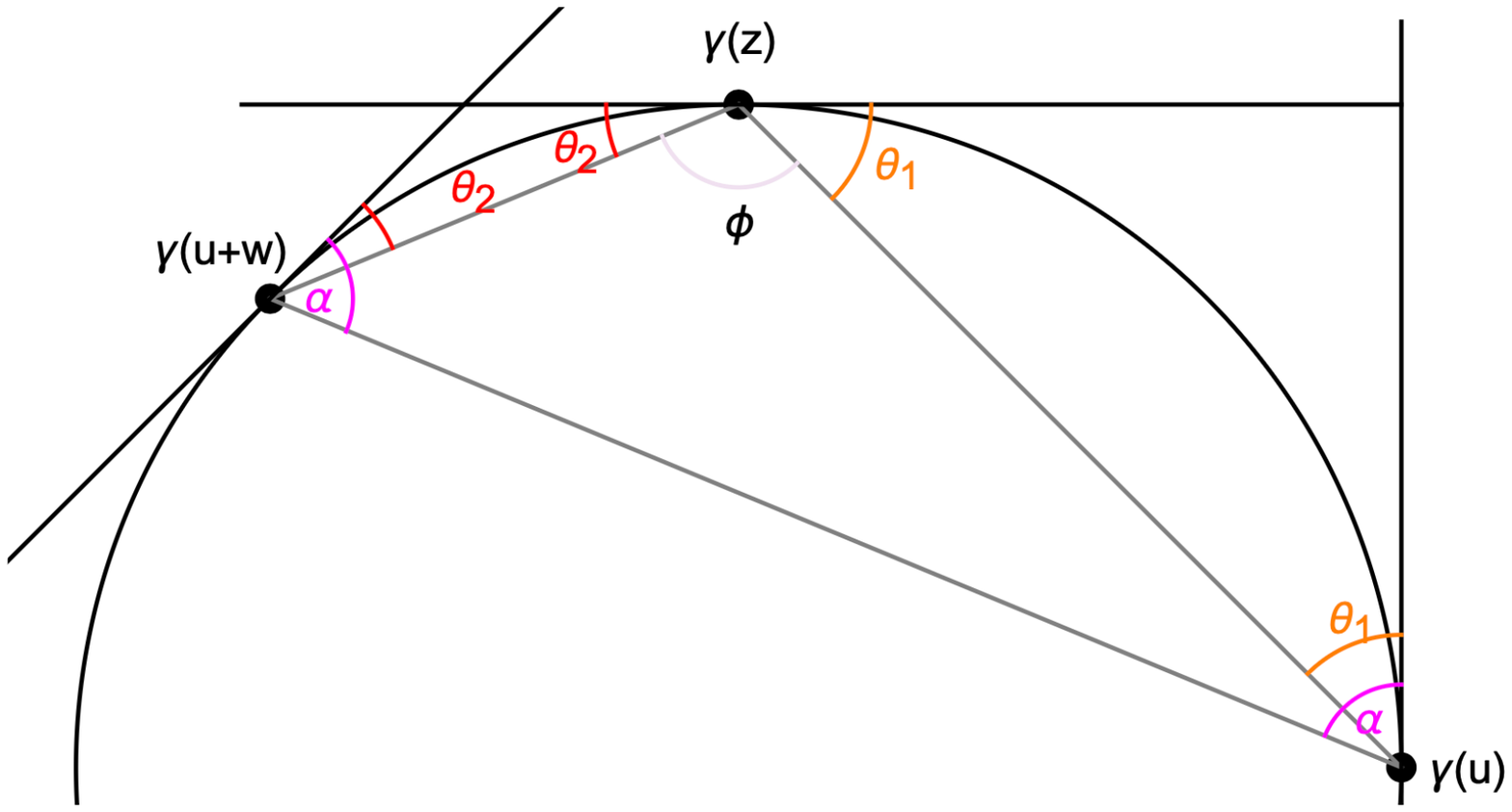}
\end{center}
    After applying some more elementary geometry, we observe
    \begin{align*}
        \pi&=
            2\alpha
            -(\theta_1+\theta_2)
            +
            \phi
            =
            \pi
        \\
        \pi
        &=
            \phi
            +
            \theta_1
            +
            \theta_2
        .
    \end{align*}
    Rearranging the terms, we conclude that $\phi=\pi-\alpha$. Note that, due to $\gamma$ being continuous and convex, $z\mapsto \phi(u,u+w,z)$ is continuous in $(u,u+w)$. Therefore, we conlude that for all $z_1,z_2\in (u,u+w)$, we have that
    \[
        \phi(u,u+w,z_1)
        =
        \phi(u,u+w,z_2)
    .
    \]
    If one of the statements holds true, we can apply \aref{lem:UmlaufWinkelSatzScharfBeiKreis} and conclude that $\Image(\gamma)$ is a circle.
\end{proof}

\begin{lemma}
	\label{lem: constant derivative of GaussMap}
	Let $\gamma\in W^{1,1}_{\mathrm{i,r}}(\Domain, \mathbb{R}^2)$.
	If $\abs{\partial_w \GaussMap_\gamma}=\pi$ for almost all $u,w$ or $\abs{\partial_u \GaussMap_\gamma}=2 \pi$ for almost all $u,w$, then $\gamma$ is a circle, parametrized by constant speed. 
\end{lemma}

\begin{proof}
	Let $\abs{\partial_w \GaussMap_\gamma}=\pi$ for almost all $u,w$.
    Using \aref{eq:DerivGaussMapW} and abbreviating $r_{\TP}\ceq r_{\TP}(\gamma)$, this implies
	\[
	    \abs{\partial_w \GaussMap_\gamma(u,w)}
        =
        \frac{1}{2} \frac{1}{r_{\TP}(u+w,u)} \abs{\gamma ^\prime (u+w)}
        =
        \pi
        \text{ for almost every }
        u,w
        .
	\]
	Substituting $x=u$ and $y=u+w$, this yields $r_{\TP}(y,x) =\frac{\abs{\gamma ^\prime (y)}}{2\pi}$ for almost all $x,y \in \Domain$. 
    Choose $y \in \Domain$ such that $\gamma ^\prime (y)$ exists and satisfies $r_{\TP}(y,x) =\frac{\abs{\gamma ^\prime (y)}}{2\pi}$ for almost all $x \in \Domain$. 
    Then by continuity of $\gamma$ and $r_{\TP}(y,\cdot)$, it holds for all $x \in \Domain$. 
    Using the supporting hyperplane theorem, we may assume that $\gamma (y)=0$, $\gamma'(y)= \alpha e_1$ and $\Image(\gamma)\subset \R\times \intervalco{0,\infty}$.
    Therefore, $\Image (\gamma )$ is contained in the following circle.
    \[
        \Image(\gamma)\subset 
        \{ p\in \R\times \intervalco{0,\infty}: \abs{p-\frac{\abs{\gamma'(y)}}{2\pi}e_2}=\frac{\abs{\gamma'(y)}}{2\pi}\}
    \]
    Since $\gamma$ is an embedding, it follows that $\Image(\gamma)$ is a circle.
    Hence, $r_{\TP}(x,y)$ is constant on $\Domain ^2$. 
    Plugging this into $\abs{\partial_w \GaussMap_\gamma}=\pi$ yields that $\abs{\gamma ^\prime}$ is constant. 
    \\
	\\
	Now suppose that $\abs{\partial_u \GaussMap_\gamma}=2 \pi$ almost everywhere. 
    First, we use the critical point equation in order to improve the a priori regularity of the curvature $\kappa_\gamma$.
    \\
	By \aref{lem:LengthOfGaussMapinU}, we observe that $\gamma$ is convex.
    In particular, the curvature vector $\vec{\kappa} _\gamma$ exists almost everywhere. 
    By contradiction, one can prove that $\gamma$ is strictly convex. 
    Let $T=\gamma'/\abs{\gamma'} \in L^\infty$. 
    A short computation yields $T=\lim_{w\rightarrow 0}\GaussMap_\gamma$.
    By convexity, $T$ is differentiable almost everywhere with $\partial_{u}T=\vec{\kappa} _\gamma \abs{\gamma '}$. 
	We now claim, that $\partial_u T= \lim_{w\rightarrow 0}\partial_u \GaussMap_\gamma$.
    Let $u\in \Domain$ be given, such that $\vec{\kappa_{\gamma}(u)}$  exists.
    One computes that
    \begin{align*}
		\lim_{w\rightarrow 0}
		\partial_u \GaussMap_\gamma
		&=
		\lim_{w\rightarrow0}
			\frac{1}{\abs{ \gamma(u+w)-\gamma(u)}}
			(
			\gamma'(u+w)-\gamma'(u)
			-
			\inner{\GaussMap_\gamma, \gamma'(u+w)-\gamma'(u)}\GaussMap_\gamma
			)
		\\
		&=
		\frac{1}{\abs{\gamma'(u)}}
			(
			\gamma''(u)
			-
			T(u)\inner{T(u), \gamma''(u)}
			)
		=
	    \vec{\kappa}_\gamma (u) \abs{\gamma'(u)} \quad \text{for almost every } u \in \Domain
        .
	\end{align*} 
	Since $\abs{\GaussMap_\gamma} \equiv 1$ and $\abs{\partial_u \GaussMap_\gamma} \equiv 2 \pi$ almost everywhere, we observe that $\GaussMap_\gamma \in W^{1,\infty}$.
    Hence, by dominated convergence $T\in W^{1,1}$.
    Furthermore, $\abs{T'}=\abs{\kappa _\gamma} \abs{\gamma '}=2 \pi$ and thus 
    \[
        T'(u)= \abs{T'(u)} \, J\,T(u)=2 \pi \,  J \,T(u),
        \text{ where }
        J=\begin{pmatrix}
            0&1\\-1&0
        \end{pmatrix}
        .
    \]
    Since $u\mapsto 2\pi \, J\,T(u)$ is continuous on $\Domain$, we obtain a continuous representative of the weak derivative $T'$.
    Hence, $T \in C^1$. 
    Using the convexity of $\gamma$ we also get
	\begin{align*}
		2 \pi =& 
		\frac{
			\abs{P^\perp _{
					\GaussMap_\gamma(u+w,u)
				} \gamma^\prime(u)}}
		{\abs{ \gamma(u+w)-\gamma(u)}}
		\, + \, 
		\frac{
			\abs{P^\perp _{
					\GaussMap_\gamma(u+w,u)
				} \gamma^\prime(u+w)}}
		{\abs{ \gamma(u+w)-\gamma(u)}}
		\\
		=&
		\frac{
			\abs{P^\perp _{T(u) } (\gamma (u+w)-\gamma (u))}}
		{\abs{ \gamma(u+w)-\gamma(u)}^2}
		\, \abs{\gamma '(u)}
		\, + \,
		\frac{
			\abs{P^\perp _{T(u+w) } (\gamma (u+w)-\gamma (u))}}
			{\abs{ \gamma(u+w)-\gamma(u)}^2}
			\, \abs{\gamma '(u+w)}	
	\end{align*}
    Rearranging the terms, one obtains the following.
    \[
		\abs{\gamma '(u+w)}
		=
		[
            2 \pi
            -
			\frac{
				\abs{P^\perp _{T(u) } (\gamma (u+w)-\gamma (u))}}
			    {\abs{ \gamma(u+w)-\gamma(u)}^2}
			\abs{\gamma '(u)}
        ] 
		\, 
		\frac{\abs{ \gamma(u+w)-\gamma(u)}^2}
			{\abs{P^\perp _{T(u+w) } (\gamma (u+w)-\gamma (u))}}
    \]
	Strict convexity of $\gamma$ implies that $\abs{P^\perp _{T(u+w) } (\gamma (u+w)-\gamma (u))}=0$ iff $w=0$. 
    Thus, fixing $u_1 \in \Domain$ such that $\gamma '(u_1)$ exists and using $T\in C^1, \gamma \in W^{1,1}_{\mathrm{i,r}}$,
    we observe that the right hand-side is continuous for $w \in (0,1)$. 
    Repeating the same procedure for $u_2 \in \Domain\setminus \{u_1\}$, such that $\gamma '(u_2)$ exists, we conclude that $\abs{\gamma'}\in C^0$. 
    Using $\gamma '= T \cdot \abs{\gamma '}$, we conclude that $\gamma \in C^1$. 
    Furthermore, this implies existence and continuity of $\vec{\kappa}_\gamma=T'/\abs{\gamma '}$ on $\Domain$. 
	By \aref{lem:LengthOfGaussMapinW}, the following holds true for all $u\in \Domain$.
	\begin{align*}
	2 \pi 
	=& 
	\int_0^1 \abs{\partial_u \GaussMap_\gamma} \dd w
	=
	\int_0^1
	\frac{
		\abs{P^\perp _{
			\GaussMap_\gamma(u+w,u)
		} (\gamma^\prime(u+w)-\gamma^\prime(u))}}
	{\abs{ \gamma(u+w)-\gamma(u)}} \dd w
	\\
	\overset{\gamma \text{ convex}}{=}&
	\int_0^1
	\frac{
		\abs{P^\perp _{
				\GaussMap_\gamma(u+w,u)
			} \gamma^\prime(u)}}
	{\abs{ \gamma(u+w)-\gamma(u)}}
	\, + \, 
	\frac{
	\abs{P^\perp _{
			\GaussMap_\gamma(u+w,u)
		} \gamma^\prime(u+w)}}
		{\abs{ \gamma(u+w)-\gamma(u)}} \dd w
		\\
	=&
	\int_0^1
	\frac{1}{2} \frac{1}{r_{\TP}[\gamma](u,u+w)} \abs{\gamma ^\prime (u)} \dd w
	\, + \, 
	\int_0^1
	\abs{\partial_w \GaussMap_\gamma} \dd w
	\\
	=&
	\int_0^1
	\frac{1}{2} \frac{1}{r_{\TP}[\gamma](u,u+w)} \abs{\gamma ^\prime (u)} \dd w
	\, + \, 
	\pi
	\end{align*}
    Rearranging the terms, we observe that
    \begin{equation}\label{eq:CurvViaIntegralOfTPRadius}
        \int_0^1
	        \frac{1}{r_{\TP}[\gamma](u,u+w)} 
        \dd w
    =
        \frac{2 \pi}{\abs{\gamma '(u)}} 
    =
        \abs{\kappa_\gamma (u)}
    .
    \end{equation}
    Let $u^\ast\in \argmin_{u\in \Domain}\abs{\kappa_\gamma(u)}$.
    By \cite[Theorem 1']{koutroufiotis1972blaschke}, $\Image(\gamma)$ is contained in the closed disc, whose boundary is the osculating circle of $\gamma$ at $u^\ast$.
	Thus,
	\[
	    \frac{1}{r_{\TP}[\gamma](u^\ast,u^\ast+w)} \geq \abs{\kappa_\gamma (u^\ast)}
        \text{ for all }
        w\in [0,1].
	\]
    Combining this observation with \aref{eq:CurvViaIntegralOfTPRadius}, we conclude that
	\[
	    \abs{\kappa _\gamma (u^\ast)}
        = 
	    \frac{1}{r_{\TP}[\gamma](u^\ast,u^\ast +w)} 
        \text{ for all } w \in [0,1]
        .
	\]
	Hence, $\Image (\gamma)$ is a circle and therefore, $\kappa_ \gamma$ is constant. 
    Inserting this into $2\pi = \abs{\kappa _\gamma} \abs{\gamma '}$, we infer that $\abs{\gamma '}$ is constant.
\end{proof}

\begin{remark}
	\label{rmk: constant derivative of GaussMap}
	If in \aref{lem: constant derivative of GaussMap} one only considers arc length parametrized curves or the geometric version $\tilde{\GaussMap}$, the proof simplifies substantially , as one can immediately conclude at $\abs{\kappa _\gamma}\equiv 2 \pi$.
    This also happens, if one assumes $\abs{\frac{1}{\abs{\gamma'(u)}}\partial_u \GaussMap_\gamma}=2\pi$ or $\abs{\frac{1}{\abs{\gamma'(u+w)}}\partial_w \GaussMap_\gamma}=\pi$.
    These scenarios arise, if one minimizes 
    \begin{align*}
        \Dirichlet_p(\gamma)
        &=
        \int_{0}^1
            \abs{\frac{1}{\abs{\gamma'(u)}}\partial_u \GaussMap_\gamma}^p
        \abs{\gamma'(u)}\dd u
        \quad \text{ or }
        \\
        \Dirichlet_p(\gamma)
        &=
        \int_{0}^1
            \abs{\frac{1}{\abs{\gamma'(u+w)}}\partial_u \GaussMap_\gamma}^p
        \abs{\gamma'(u+w)}\dd w
    \end{align*}
    by using Jensen's inequality and \aref{lem:LengthOfGaussMapinU} (respectively \aref{lem:LengthOfGaussMapinW}).
\end{remark}